\documentclass[12pt,a4paper]{article}
\usepackage{lipsum}
\usepackage{authblk}
\usepackage{amsmath,a4}  
\usepackage[utf8]{inputenc}
\usepackage{graphicx,enumerate,color,amsfonts,amsthm,amsmath,amssymb,mathrsfs,marginnote,dsfont,textcomp,tikz,marginnote,subfigure,pdflscape,multirow,booktabs,float,enumitem,scrtime,currfile,titleps,hyphenat}
\usepackage[colorlinks,
    linkcolor={red!60!black},
    citecolor={black},
    urlcolor={blue!80!black}]{hyperref}
\usepackage{tikz,pgf,pgffor,pgfplots}
\tikzset{>=latex}
\usetikzlibrary{calc,patterns,decorations,intersections,decorations.pathmorphing,fit,backgrounds,arrows,shapes,snakes,automata,petri,positioning}
\usepgfmodule{shapes,plot}

\newtheorem{theorem}{Theorem}[section]

\theoremstyle{definition}

\theoremstyle{remark}
\newtheorem*{remark}{Remark}

\makeatletter
\newcommand{\setword}[2]{%
  \phantomsection
  #1\def\@currentlabel{\unexpanded{#1}}\label{#2}%
}
\makeatother

\newcommand{\dd}{\textnormal{d}}

\newcommand{\Id}{I^{\textnormal{DATA}}}

\newcommand{\DFE}{\textnormal{DFE}}
\newcommand{\EE}{\textnormal{EE}}

\newcommand{\R}{\mathbb{R}}
\newcommand{\RM}{\mathcal{R}}
\allowdisplaybreaks
\usepackage[top=2cm, bottom=2cm, left=2cm, right=2cm]{geometry}
\usepackage{fancyhdr}
\newpagestyle{main}{
\setheadrule{.4pt}
\sethead{}
        {}
        {\sectiontitle}
}
\renewenvironment{abstract}{%
\hfill\begin{minipage}{0.95\textwidth}
\rule{\textwidth}{1pt}}
{\par\noindent\rule{\textwidth}{1pt}\end{minipage}}
\makeatletter
\renewcommand\@maketitle{%
\hfill
\begin{minipage}{0.95\textwidth}
\vskip 2em
\let\footnote\thanks 
{\Large \bf \@title \par }
\vskip 1.5em
{\large \@author \par}
\end{minipage}
\vskip 1em \par
}
\makeatother
\begin{document}
%
\title{\nohyphens{Learning the seasonality of disease incidences from empirical data}}
\author[1]{Karunia Putra Wijaya}
\author[2,$\ast$]{Dipo Aldila}
\author[3]{Luca Elias Sch{\"a}fer}
\affil[1]{\small\emph{Mathematical Institute, University of Koblenz, 56070 Koblenz, Germany}}
\affil[2]{\small\emph{Department of Mathematics, Universitas Indonesia, 16424 Depok, Indonesia}}
\affil[3]{\small\emph{Department of Mathematics, Technical University of Kaiserslautern, 67663 Kaiserslautern \\Germany}}
\affil[$\ast$]{Corresponding author. Email: \href{mailto:aldiladipo@ui.sci.ac.id}{aldiladipo@ui.sci.ac.id}}
\maketitle
\begin{abstract}
{\textbf{Abstract}}: Investigating the seasonality of disease incidences is very important in disease surveillance in regions with periodical climatic patterns. In lieu of the paradigm about disease incidences varying seasonally in line with meteorology, this work seeks to determine how well standard epidemic models can capture such seasonality for better forecasts and optimal futuristic interventions. Once incidence data are assimilated by a periodic model, asymptotic analysis in relation to the long-term behavior of the disease occurrences can be performed using the classical Floquet theory, which explains the stability of the existing periodic solutions. For a test case, we employed an IR model to assimilate weekly dengue incidence data from the city of Jakarta, Indonesia, which we present in their raw and moving-average-filtered versions. To estimate a periodic parameter toward performing the asymptotic analysis, eight optimization schemes were assigned returning magnitudes of the parameter that vary insignificantly across schemes. Furthermore, the computation results combined with the analytical results indicate that if the disease surveillance in the city does not improve, then the incidence will raise to a certain positive orbit and remain cyclical.
~
\vspace*{-5pt}
\end{abstract}

\section{Introduction}

Investigating the seasonality of disease incidences plays a fundamental role in the detection of future outbreaks, and thus in the provision of economically viable control interventions. In the case of vector-borne diseases, previous studies suggested that the incidence is sensitive to the behavior of climatic parameters \cite{BDG2002,Bab2003,MMC2017,SEM2017}. The rationale behind this is that the aforementioned parameters, including water precipitation, temperature, and air humidity, have a major influence on the growth of virus-carrying vectors. Let us look, for example, at the case of dengue. A high temperature, on the one hand, prolongs the life of mosquitoes \emph{Aedes aegypti} and shortens the extrinsic incubation period of the dengue virus, which then increases the number of living infected mosquitoes \cite{MMC2017,LPF2011}. On the other hand, high rainfall induces the development of breeding sites for mosquitoes \cite{BVV2015,PWG2007}, while high wind speeds disrupt the development of the aquatic phases on leaf blades, light plastics, or other unstable sites that further induce negative correlation with the incidence of dengue \cite{CBL2013,RRJ2011}. The message here is that one can achieve a better understanding of the seasonality of disease incidences in some regions where the corresponding climatic parameters fluctuate periodically. 

One idea to approach fluctuating data and generate some forecasts is to develop epidemic models in which several biological properties can be adjusted. These properties include birth rate, recruitment, death rate, rate of infection, recovery, loss of immunity, incubation, temporary cross-immunity, secondary infection, age structuring, spatial heterogeneity, and internal (microbiological) processes \cite{GAB2017,CGS2017,RMS2015,PRT2017,WSS2017}. A promising feature of epidemic models is their accessibility to further development, i.e.,\ by including more biological properties, to further increase realism. As far as assimilating almost-periodic data is concerned, one can designate some involved parameters in the model to be periodic by an argument that can be correlated with some periodical climatic parameters \cite{Cus1998,HE2007,HC1997,IMN2007,SLP2007,SLY2003,WGS2016}. Decision-making comes into play when prior knowledge about the long-term behavior of the assimilating solution of the model is gained. This asymptotic analysis is the part where the \emph{basic reproductive number} plays a significant role. The basic reproductive number is defined as the number of secondary infections that happen when a single infective individual comes into a completely susceptible population during the infection period \cite{DHM1990,DH2000,DHR2009}. Typical asymptotic analysis results in the following statement: if the basic reproductive number is greater than one, then the disease will persist for a longer period of time; if the number is less than one, then the disease will die out for a much longer period of time \cite{DH2000,EV1998,SA2015}. This means that the larger the basic reproductive number, the higher the probability of the population remaining in an infection circle in the future. 


In the case of epidemic models where all the involved parameters are constant, determination of the basic reproductive number is straightforward. For models where the healthy and infective subpopulations are separable, the so-called \emph{next-generation method} has widely been used to generate the basic reproductive number by associating it with the local stability of the \emph{disease-free equilibrium} (DFE) and the \emph{endemic equilibrium} (EE) \cite{DHR2009,WS2002}. For models with periodic parameters, equilibria are no longer present, but periodic solutions are. Here, the determination of the true basic reproductive number becomes more challenging, since it should now be associated with the local stability of the trivial and nontrivial periodic solutions \cite{WGS2016,TW2014}. In lieu of the asymptotic analysis of nonautonomous periodic models, the local stability of periodic solutions can be understood with the aid of the Floquet theory. It addresses some conditions for which a periodic solution is locally asymptotically stable, which can bear a relation to the basic reproductive number from the autonomous counterpart or to a completely different basic reproductive number.

Some epidemic models have previously been introduced together with periodic parameters, with \cite{RMS2015,WGS2016,TW2014} or without the analysis of periodic solutions \cite{MBK2011,MKR2013,AHB2013}. One apparent issue from those with the analysis is that their parameters are never calibrated using real field data. Some other models brought along time-varying parameters to pronounce a perfect match between their solutions and field data, however requiring further effort to determine the prediction of the value of the optimal parameters at each time in the next time window \cite{GAB2017,WMN2012,ASH2016}. Unless sufficient forecasts of some extrinsic factors, e.g.,\ climatic parameters, are known, the last strategy potentially disregards some asymptotic analysis. Relying on a nonautonomous periodic model, our contribution lies heavily in combining the results from data assimilation and those from Floquet theory to perform a simple decision-making process. We use the final result to not only detect future outbreaks of a disease but also predict the unforeseen behavior of the incidence trajectory in the long run.

\section{Results from Floquet theory}\label{sec2}
Floquet theory has initially been proposed in \cite{Flo1883} to analyze the local stability of periodic solutions of the following linear differential equation with periodic coefficients
\begin{equation}\label{eq:linear}
\dot{z}=A(t)z,\quad z(0)=z_0,\ A(t)\text{ is } \sigma\text{-periodic}.
\end{equation}
The last equation appears in a multitude of applications, including biology \cite{RSE2011,BP2012}, chemistry \cite{DFB2014}, quantum physics \cite{Lun2000,Moo1993} and economics \cite{Hei2015}. According to \cite{Flo1883}, the important finding from the linear system~\eqref{eq:linear} is that the fundamental matrix of the system $Z(t)$ can be represented as 
\begin{equation*}
Z(t) = P(t)\exp\left(Qt\right),
\end{equation*}
where $P(t)$ is a continuously differentiable $\sigma$-periodic matrix and $Q$ is a constant matrix. All the eigenvalues of $Q$ are what we know as the \emph{Floquet exponents}; meanwhile, all the eigenvalues of $Z(\sigma)$ are the \emph{Floquet multipliers} that correspond to the system matrix $A(t)$. Under the assumption of nonsingularity over the system matrix, analyzing the Floquet exponents in relation to the local stability of the trivial periodic solution can be tricky if the system is large. One practical aid to help with this is by analyzing the structure of the system matrix; see \cite{TW2014} for details. 

Our study here is devoted to analyzing the existence and local stability of periodic solutions of the following nonlinear system
\begin{equation}\label{eq:nonlinear}
\dot{x}=f(t,x)=g(x) + \delta h(t,x),\quad x(0)=x_0,\ h\text{ is }\sigma\text{-periodic}.
\end{equation}
From the last system, $\delta$ denotes some constant, $h$ is assumed to be continuous in time, while both $g$ and $h$ are assumed to be continuously differentiable in state with bounded derivatives, evoking Lipschitz continuity and therefore the global existence of a unique solution. For typical nonlinear nonautonomous periodic systems, the existence of periodic solutions can always be associated with the equilibrium of the autonomous part of the system. Details are provided by the following theorem.

\begin{theorem}\label{thm:existperiodic}
Consider the nonlinear system~\eqref{eq:nonlinear}. Let $x^{\ast}$ be an equilibrium point of the autonomous part $\dot{x}= g(x)$. If all eigenvalues $\lambda$ of the Jacobian $\nabla_{x} g(x^{\ast})$ hold $\lambda\notin \frac{2\pi}{\sigma} i\mathbb{Z}$ where $i$ and $\mathbb{Z}$ denote the imaginary number and the set of integers, respectively, then there exist a neighborhood $U(x^{\ast})$ and $\alpha$ such that for every $|\delta|<\alpha$, there exists a $\sigma$-periodic solution of the system~\eqref{eq:nonlinear} around $x^{\ast}$ with a unique initial condition $x_0= x_0(\delta)\in U(x^{\ast})$.
\end{theorem}

\begin{proof}
See Appendix~\ref{app:exist}. 
\end{proof}

Observe that this existence theorem only requires looking at the eigenvalues of the Jacobian at an equilibrium, which raises technical difficulties once the explicit formulation of the equilibrium cannot be revealed. Now, in the case that a periodic solution exists, its stability can be analyzed by investigating the dynamics of the error between the actual solution and the periodic solution. The periodic solution apparently contains the linear part that resembles the linear system~\eqref{eq:linear} whose stability is determined by the Floquet exponents that correspond to the system matrix. Before going into more details of the stability, we briefly recall a much simpler mathematical definition of the stability of a periodic solution in the sense of Lyapunov. Let $x=x(t,0,x_0)$ and $\phi$ be the solution at time $t$ that was initiated from $0$ with the corresponding initial condition $x_0$ and a periodic solution of the nonlinear system~\eqref{eq:nonlinear}. The solution $\phi$ is called \emph{locally asymptotically stable} if there exists a positive number that bounds $\lVert x(0)-\phi(0)\rVert$ such that the trajectory of $\lVert x-\phi\rVert$ is bounded by a decreasing function that converges to zeros as the time $t$ tends to infinity. Conditions that certify the local stability are then proposed as in the following theorem.

\begin{theorem}\label{thm:stableperiodic}
Let $y:= x-\phi$ be some measure of the error between the actual and periodic solutions and $m(t, y) :=  f(t, y+\phi)- f(t,\phi)-A(t) y$ where $A(t):=\nabla_{ x} f(t,\phi)$. If the following conditions hold:
\begin{enumerate}[label= \normalfont ({A}\arabic*)]
\item $\lVert m(t, y)\rVert\leq M\lVert y(t)\rVert$ for some constant $M$ and all $t\in [0,\infty)$;\label{item:m}
\item all the Floquet exponents that correspond to $A(t)$ lie in the open left-half plane in $\mathbb{C}$; \label{item:floquet}
\end{enumerate}
then $\phi$ is locally asymptotically stable. 
\end{theorem}

\begin{proof}
See Appendix~\ref{app:stable}.
\end{proof}

\begin{remark}
Note that a similar stability result can be obtained for which the condition~\ref{item:m} is replaced by $\lVert m(t, y)\rVert\leq M\lVert y(t)\rVert^2$. The proof follows similarly as in Appendix~\ref{app:stable} by excluding the appearance of $\lVert y^{\ast}\rVert$.
\end{remark}


Observe that the stability theorem (Theorem~\ref{thm:stableperiodic}) requires the explicit formulation of an investigated periodic solution. Difficulties in determining the Floquet exponents remain if such an explicit formulation is not known. In fact, it is hard to reveal an explicit nontrivial periodic solution of most standard SIR-like epidemic models with periodic parameters even though the corresponding nontrivial equilibrium is known. Later, we use a similar argument for the proof of Theorem~\ref{thm:stableperiodic} to investigate the stability of a nontrivial periodic solution even without knowing its explicit formulation.


\section{SIRUV model and quasi-steady-state approximation}\label{sec3}

The SIR model was proposed by Kermack and McKendrick in \cite{KM1927}. The model was designed to capture the dynamics of three subpopulations of humans over time, each of which indicates the current status of infection. A person is categorized as: susceptible (S) if he or she is healthy but at risk of becoming infected by the disease; infective (I), if he or she obtains pathogens inside of his or her body and are able to transmit them to other people; recovered (R) if he or she is cured from and confers lifelong immunity against the disease. Nowadays SIR-like models are developed that also include the dynamics of the vector \cite{CGS2017,EV1998,AGS2013,MKK2017,MNR2017}. As for our current discussion, we use the following SIRUV model \cite{RRJ2011,RMS2015,RAS2013} for measuring the epidemicity of dengue diseases:
\begin{alignat*}{2}
\dot{S}&=\mu(N-S) - \frac{\beta}{M} SV + \kappa R,   \quad  &&\dot{I}= \frac{\beta}{M} SV  - (\gamma+\mu)I,\quad 
\dot{R}= \gamma I - (\mu+\kappa)R,\\
\dot{U} &= \Lambda - \frac{\rho}{N} U I - \theta U, \quad 
 &&\dot{V}=\frac{\rho}{N} U I  - \theta V.
\end{alignat*}

In the model, we indicated that any newborns of human and mosquito are always susceptible. We also indicated that any susceptible human experiences a very short incubation period, allowing him or her to move to the infective compartment immediately after a contact with an infective vector. The parameters in the human compartments $\mu$, $\beta$, $\gamma$, and $\kappa$ denote the death rate, the infection rate, the recovery rate, and the loss-of-immunity rate indicating a simplest representation of the so-called \emph{temporary cross-immunity} between sequential dengue infections \cite{MBK2011,MKR2013,LR2013,LR2016,Sab1952}, respectively. The variables $U,V$ measure the number of susceptible and infective mosquitoes. In a similar sense, $\Lambda$, $\rho$, and $\theta$ denote the constant recruitment rate, the infection rate, and the death rate of mosquitoes, respectively. The total populations $N,M$ are simply $N:=S+I+R$ and $M:=U+V$. We assume that all the involved parameters are positive.

What we can see from the model is that the lifetime period of humans is significantly greater than that of mosquitoes, i.e.,\ $\mu^{-1}\gg \theta^{-1}$, from which both dynamics occur on disparate timescales. As a consequence of this imbalance, the mosquito population equilibrates much faster than the human population. In such a situation, we refer to the mosquito population as the ``fast population'' and to the human population as the ``slow population.'' The so-called quasi-steady-state approximation (QSSA) engineer timescale separation to show that the dynamics of the mosquito population over a long period with a much large timescale behaves like its equilibrium \cite{RAS2013,SS1989,BDS1996}. Looking back at our model, this means
\begin{equation*}
U(I)\approx \frac{\Lambda N}{\theta N+\rho I},\quad V(I)\approx \frac{\rho\Lambda I}{\theta(\theta N + \rho I)}, \quad \text{and} \quad \frac{V(I)}{M}\approx \frac{I}{I+(\theta\slash \rho)N}.
\end{equation*}
Proposition of the only argument $I$ in the last formulation results from the fact that $\dot{N}=0$ in our model; therefore, $N$ can be seen as another constant. Accordingly, the dynamics of the fast population is no longer displayed and together with the constant human population $N$, the original SIRUV system can be reduced to a lower-dimensional system 
\begin{subequations}
\begin{align}
\dot{I} &= \beta (N-I-R)\frac{I}{I+(\theta\slash\rho)N} - (\gamma+\mu)I,\\
\dot{R} &= \gamma I- (\mu+\kappa)R.
\end{align}
\end{subequations}
Unless the time window of observation with the model is sufficiently large, i.e.,\ $\gg \theta^{-1}$, the last IR model may not be realistic in modeling the dengue transmission by excluding vector dynamics. Advantageously, not only is the IR model reliable at representing the original SIRUV model on a large time window, but it is also simpler for data assimilation.

\subsection{Uniperiod analysis}
With information sharing among experts, the parameters $\mu$, $\gamma$, and $\kappa$ can most likely be specified forthwith. The other two parameters $\nu:=\theta\slash\rho$ and $\beta$ require a review of the historical infections, and therefore of the behavior of given field data. To capture the periodicity of field data, we designated
\begin{equation*}
\beta(t) = \alpha + \delta p(t),
\end{equation*}
where $p$ and its antiderivative $\int_0^tp(s)\,\text{d}s$ are continuous $\sigma$-periodic functions and $\nu$ as to be constant; see the specific choice $\nu=1\slash 2$ in \cite{RAS2013,GAB2017}. It is worth noting that $p$ has to be chosen in such a way that its antiderivative vanishes at $0,\sigma,2\sigma,\cdots$. By the last specification we actually force the data to possess exactly one significant period $\sigma$ or many periods whose least common multiple is $\sigma$. In addition, by the last specification we rediscover the nonlinear system~\eqref{eq:nonlinear}, which is ready for the analysis. The vector field of the IR model clearly is continuously differentiable in state for any positive bounded $N$ and continuous in time, which evokes the global existence of a unique solution. In addition, the product between the vector field and the outward normal to the boundary of the nonnegative orthant in the two-dimensional Euclidean space evaluated at all the points in each boundary returns in nonpositive values. This means that once the initial condition of the system is nonnegative, the solution trajectory would remain nonnegative for all time in the future. From the autonomous part
\begin{align*}
\dot{I} &= \alpha (N-I-R)\frac{I}{I+\nu N} - (\gamma+\mu)I\\
\dot{R} &= \gamma I- (\mu+\kappa)R,
\end{align*}
we can easily reveal the DFE and the EE
\begin{align*}
\DFE=(0,0) \quad \text{and} \quad \EE=\left(\frac{N(\mu+\kappa)(\gamma+\mu)\nu\left(\frac{\alpha}{(\gamma+\mu)\nu}-1\right)}{(\alpha+\mu)(\gamma+\mu+\kappa)+\gamma\kappa},\frac{N\gamma(\gamma+\mu)\nu\left(\frac{\alpha}{(\gamma+\mu)\nu}-1\right)}{(\alpha+\mu)(\gamma+\mu+\kappa)+\gamma\kappa}\right),
\end{align*}
where the existence of the EE from the biological point of view applies if the basic reproductive number $\RM_0:=\alpha\slash (\gamma+\mu)\nu>1$. Since both equilibria are explicit, we immediately obtain the following result based on Theorem~\ref{thm:existperiodic}.

\begin{theorem}\label{thm:existIR}
There always exists a trivial periodic solution $(I,R)=(0,0)$ and there exists a nontrivial periodic solution around $\EE$ providing that $\mathcal{R}_0>1$.
\end{theorem}

\begin{proof}
See Appendix~\ref{app:existIR}.
\end{proof}

We observe that the trivial periodic solution is explicit, whereas the nontrivial periodic solution is not. Using a similar argument for the proof of Theorem~\ref{thm:stableperiodic}, we were able to find the following sufficient conditions for the stability.

\begin{theorem}\label{thm:stableIR}
The following assertions hold true.
\begin{enumerate}[label=\normalfont ({B}\arabic*)]
\item The trivial periodic solution is locally asymptotically stable if the basic reproductive number $\RM_0<1$.\label{item:trivial}
\item The existing nontrivial equilibrium $\phi$ around \EE~is locally asymptotically stable if \label{item:nontrivial}
\begin{equation*}
\RM_{\max}:=\RM_0\max_{0\leq t\leq \sigma}\Phi(t)<1,
\end{equation*}
where $\Phi(t):=-\frac{\nu\phi_1(t)}{\phi_1(t)+\nu N} + \frac{\nu(N-\phi_1(t)-\phi_2(t))}{\phi_1(t)+\nu N} \left(1-\frac{\phi_1(t)}{\phi_1(t)+\nu N}\right)$.
\end{enumerate}
\end{theorem}

\begin{proof}
See Appendix~\ref{app:stableIR}.
\end{proof}

\begin{remark}
Note that it is not possible to calculate the threshold parameter $\RM_{\max}$ without knowing the explicit formulation of $\phi$. However, thanks to the periodicity of $\phi$, its numerical value can be traced when the model is realized within a sufficiently large time window. By this trace, finding $\RM_{\max}$ becomes completely numerical. One should note that the maximum of $\Phi$ is achieved if and only if $\phi\equiv 0$, which can never be the case by the nontriviality of $\phi$. Therefore, it always holds that $\RM_{\max}<\RM_0$.
\end{remark}

We investigate dengue transmission from the city of Jakarta, Indonesia, as a specific application of the previous analytical findings regarding the IR model. By assimilating collected data of dengue incidences from the city using the IR model, we expect to discover the optimal values of $\beta$ and $\nu$ based on the data. A successful completion of this task allows us to "judge'' what happens to the number of incidences in the long run according to Theorems~\ref{thm:existIR} and~\ref{thm:stableIR}.

\section{Practical instance related to decision-making}\label{sec4}

Jakarta shares warm and humid weather conditions that are typical for most of the cities in Indonesia. Through its position close to the equator, the temperature ranges between (23.9, 23.83) and (31.8, 31.38) degrees Celsius (from the stations "Observatory'' and "Tanjung Priok'', respectively) all year round \cite{NOA2017}, which is favorable for the development of Aedes mosquitoes \cite{MBB2016}. Despite being overwhelmed by high-rising skyscrapers, Jakarta still has a number of slum areas and less-maintained rivers, and is in immediate contact with a warm calm northern sea, which makes its northern coastal areas prone to flooding \cite{TEM2016}. The nature of the tropical monsoon also gives Jakarta a dry season featuring sunlight with average visibility (4.4, 4.86) and highest temperatures between June and September, and a wet season featuring downpours with average precipitation (0.6785, 0.6568) and lowest temperatures between October and May \cite{NOA2017}. 

Jakarta has been considered as one of the most endemic regions for dengue fever in Indonesia since 1960s \cite{KUK2014}. The data we present here show that the outbreak falls between April and July. These data were collected from 154 hospitals in Jakarta by the Epidemiology Department, Jakarta Health Office, consisting of daily cases from six subregions of Jakarta by separate illnesses: suspected dengue fever, dengue fever, dengue hemorrhagic fever, and dengue shock syndrome. A person with suspected dengue is a patient that already has symptoms such as fever, headache, and some rash on the body but has not yet been confirmed with a blood test. The six subregions include Central Jakarta, North Jakarta, East Jakarta, South Jakarta, West Jakarta, and Thousand Islands. The twenty-four datasets depict the dengue incidences simultaneously from January 2008 to June 2017. One should note that access to the data was restricted and was granted upon personal request. For reasons of avoiding bias due to existing holidays and making the data assimilation program faster, we accumulated the daily data into weekly data. The resulting weekly data consist of 505 data points as shown in its raw and filtered versions in Fig.~\ref{fig:data} from which the first 473 points serve as the training data for the assimilation. It is worth noting that during the observation time window, moderate El Ni\~{n}o events occurred between 2009 and 2010 \cite{IRI2014} and similar temperature anomalies occurred on July--October 2015 \cite{KKO2016}. The impacts of El Ni\~{n}o have always been associated with warmer weather conditions around the northern coastal areas, which apparently improved the cycle and spread of Aedes mosquitoes and accelerated the replication of dengue virus \cite{MMC2017,LPF2011}. This might be considered as a logical reason behind the greater incidences during or slightly after those periods of time; see Fig.~\ref{fig:data}.

\begin{figure}[ht]
\centering
\includegraphics[width=13.33cm,height=5cm]{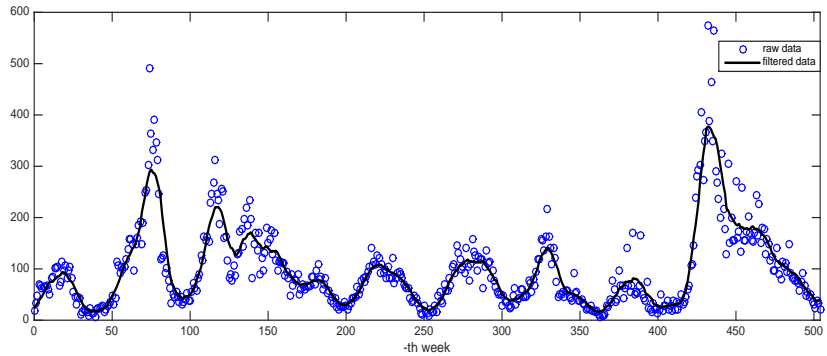}
\label{fig:data}
\caption{Weekly data of dengue incidence in the city of Jakarta by mixed illnesses: raw data (blue) and moving-average-filtered data after compensating delay (black). The moving window to use consisted of 13 weeks (quarter of a year).}
\end{figure}

\begin{figure}[htbp!]
\centering
\subfigure[]{
\includegraphics[width=8.2cm,height=3.2cm]{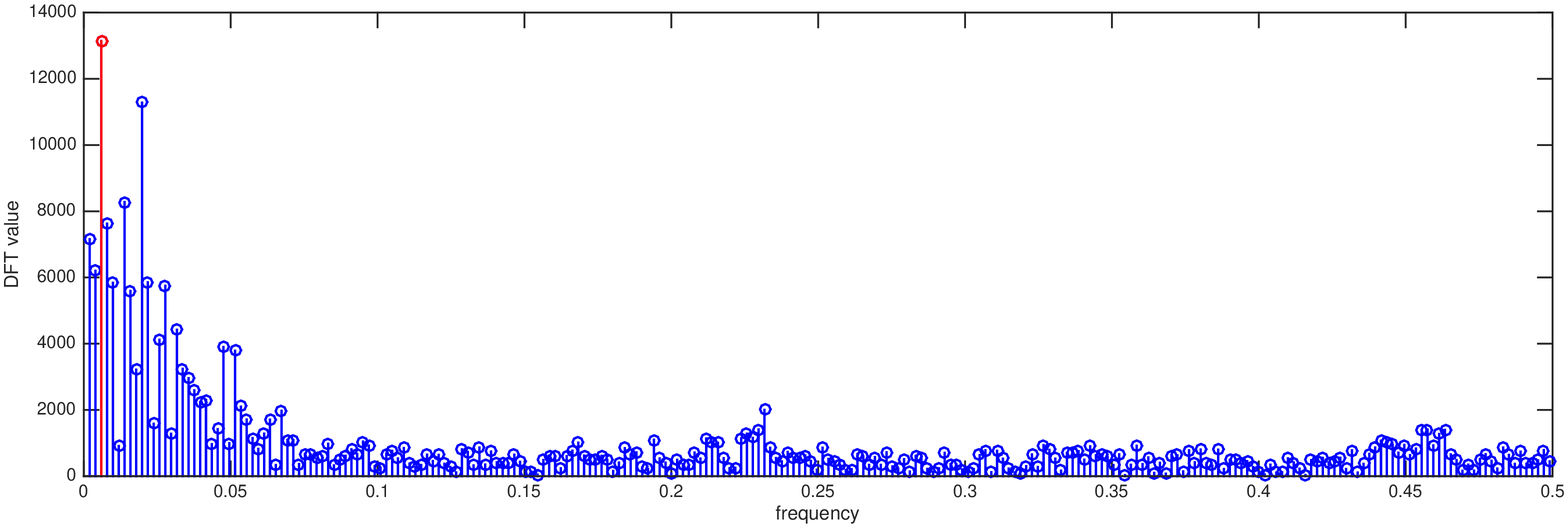}
}
\hfill\hspace*{-0.5cm}
\subfigure[]{
\includegraphics[width=8.2cm,height=3.2cm]{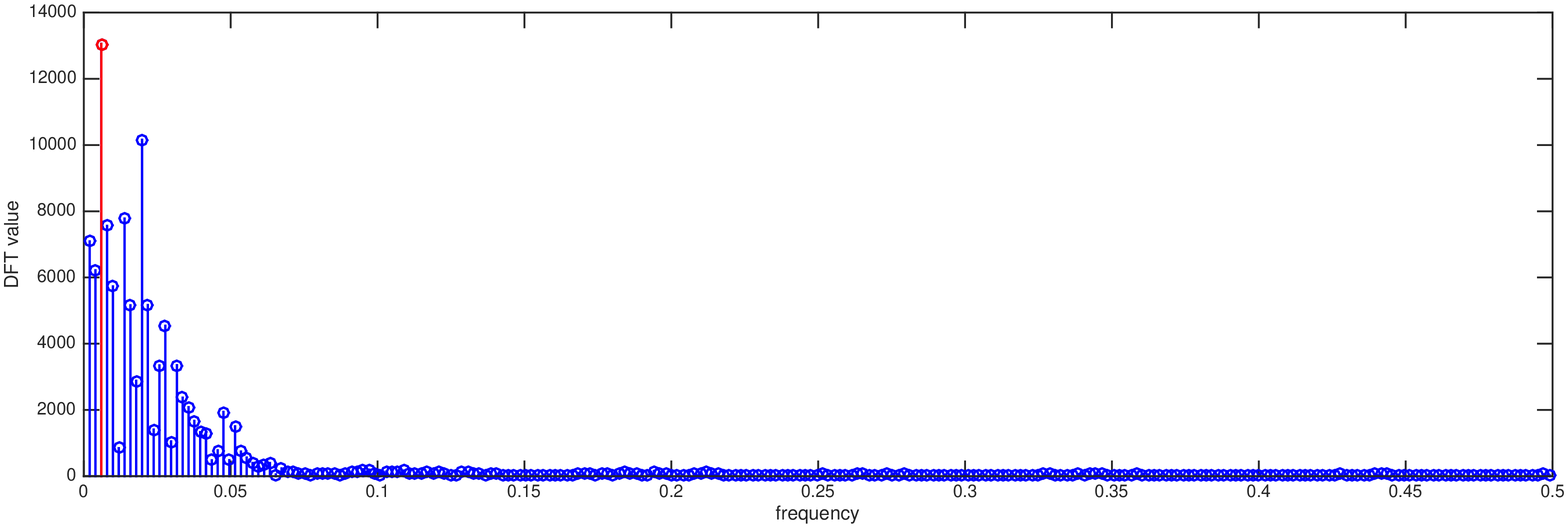}
}
\caption{Fourier transformation over the raw data (a) and the filtered data (b): the most significant frequencies appear in red.}
\label{fig:fftdata}
\end{figure}

In the interest of designing the infection rate $\beta$, we performed Fourier transformation over the data as indicated in Fig.~\ref{fig:fftdata} and extracted the most significant frequencies. Preliminary analysis using Fourier transformation shows marked frequencies 0.005941, 0.017822, and 0.0237624 for both raw and filtered data, returning the periods 168.32, 56.11, and 42.08 weeks, respectively. The second significant period indicates that dengue outbreaks have been reemerging almost every 56.11 weeks in Jakarta since January 2008, which resemble annual events. This can also be explained from Fig.~\ref{fig:data} that shows that there has been a single outbreak every year. Furthermore, the local government has implemented several intervention strategies to reduce the burden of infections including regular removal of mosquito breeding sites, dissemination of temephos to kill mosquito larvae, fumigation inside and outside human dwellings, orders to wear repellents and long clothing that help prevent exposure to mosquitoes, and quite recently, since October 2016, vaccination \cite{San2016}. The data we presented merely show that those strategies were either less effective or passively implemented, as the incidence was shown to revolve over time. 

\subsection{Data assimilation schemes}
Let us indicate the first week and the last week in the training data as $0$ and $T$, respectively. For reasons of well-posedness, let us assume that the discrete dataset $\Id$ has been interpolated to occupy the continuous domain $\Omega:=[0,T]$. For the first step, we designated 
\begin{equation*}
\beta(t) = \alpha + \delta\cos\left(2\pi\omega t\right),\quad\text{where }\omega=1\slash\sigma.
\end{equation*}
Our objective is now minimizing the metric between solution $I$ from the IR model and the data $\Id$, but preventing $\beta$ from being arbitrarily large. This task reads as
\begin{equation}\label{eq:optim}
\begin{aligned}
& \underset{(\alpha,\delta,\omega,\nu,I_0,R_0)\in\R^6}{\text{minimize}}
& &\frac{1}{2}\int_{\Omega}\left(I-\Id\right)^2\,\text{d}t + \frac{\lambda}{2} \lVert (\alpha,\delta,\omega,\nu)\rVert^2,\\
& \text{subject to}
& & \dot{I} = \beta (N-I-R)\frac{I}{I+\nu N} - (\gamma+\mu)I,\\
&
& & \dot{R} = \gamma I- (\mu+\kappa)R.
\end{aligned}
\end{equation}
Recall that $\nu=\theta\slash\rho$ was originated from the SIRUV model, and $I_0,R_0$ are the initial conditions for $I$ and $R$, respectively. The regularization parameter $\lambda$ adjusts the desired values of the parameters: the larger it is, the smaller the values of the parameters. For the performance comparison, we further classified the scheme~\eqref{eq:optim} into four schemes, i.e.,\ by adding further information. This classification can be seen from Fig.~\ref{fig:problems}.

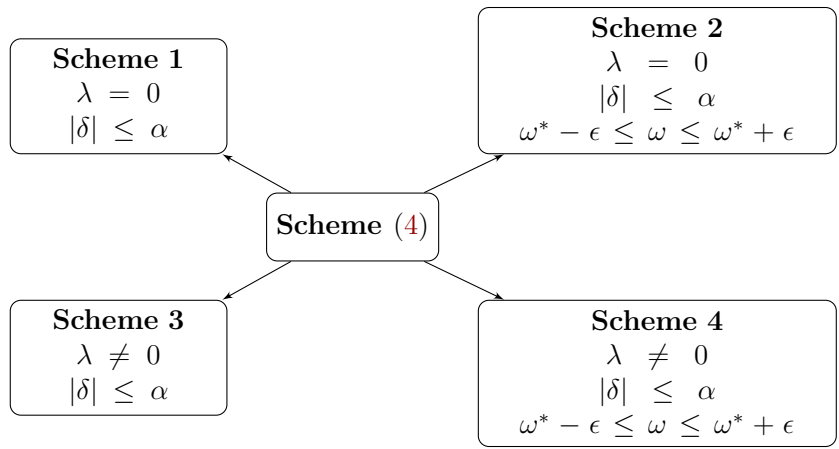
\begin{figure}[ht!]
\centering
\begin{tikzpicture}[scale=0.9, transform shape, >=latex']
        \tikzset{block/.style= {draw, rectangle, rounded corners, align=center,minimum width=2cm,minimum height=1cm},
        }
        \node [block]  (start) {\textbf{Scheme~\eqref{eq:optim}}};
        
        \node [block, above left = 0.8cm of start, text width= 7em] (A){\vbox{
        \setword{\textbf{Scheme~1}}{word:prob1}\\
        $\lambda=0$\\
	$|\delta|\leq \alpha$
}};
        \node [block, above right = 0.8cm of start, text width= 12em] (B){\vbox{
        \setword{\textbf{Scheme~2}}{word:prob2}\\
        $\lambda=0$\\
	$|\delta|\leq \alpha$\\
	$\omega^{\ast}-\epsilon\leq \omega\leq \omega^{\ast}+\epsilon$
        }};

	\node [block, below left = 0.8cm of start, text width= 7em] (C){\vbox{
	\setword{\textbf{Scheme~3}}{word:prob3}\\
	$\lambda\neq 0$\\
	$|\delta|\leq \alpha$
        }};
        
        \node [block, below right = 0.8cm of start, text width= 12em] (D){\vbox{
        \setword{\textbf{Scheme~4}}{word:prob4}\\
	$\lambda\neq 0$\\
	$|\delta|\leq \alpha$\\
	$\omega^{\ast}-\epsilon\leq \omega\leq \omega^{\ast}+\epsilon$
        }};

        \path[draw, ->]
            (start) edge (A)
            (start) edge (B)
            (start) edge (C)
            (start) edge (D) ;
    \end{tikzpicture}
\label{fig:problems}
\caption{Classification of the original scheme~\eqref{eq:optim} into four schemes.}
\end{figure}

The additional constraint on the frequency comes from the idea that the frequency of the infection rate $\beta$ should be similar to that of the data $\omega^{\ast}$, which is found from the Fourier transformation. The fact that some significant frequencies gather in the neighborhood of $\omega^{\ast}$ leads to the idea of letting $\omega$ discover its optimal value within the neighborhood $[\omega^{\ast}-\epsilon,\omega^{\ast}+\epsilon]$. Note that by the zero regularization parameter $\lambda$, the optimization problem~\eqref{eq:optim} allows the intercept $\alpha$, the amplitude $\delta$, and $\nu$ to be arbitrarily large. In addition, the constraint $|\delta|\leq\alpha$ aims at keeping $\beta$ nonnegative. 

\subsection{Numerical results}
Note that the optimization problem~\eqref{eq:optim} is an infinite-dimensional problem. Interestingly, in the problem, the initial states $I_0,R_0$ are treated as parameters whose optimal values have to be found to minimize the squared error between the solution of the IR model and the measured data over the whole time window $\Omega$. To solve the problem we use the so-called \emph{direct method} that is based on the principle of "discretize then optimize". According to this principle, the problem~\eqref{eq:optim} is to be fully discretized onto the discrete time domain on which the data were originally defined, returning a nonlinear programming (NLP) problem. Accordingly, the whole discrete states $I,R$ and the involved parameters $\alpha,\delta,\omega,\nu$ are treated as the dynamic variables in the resulting NLP problem. The discrete states now must satisfy a discrete version of the IR model whose "density'' depends on the discretization scheme taken. Either one-step or multi-step method is considered "sparse'' if it has relatively low order, but "dense'' if it has relatively high order. To avoid encountering instability in using the solution of the IR model to depict stiff curves from the data, it is always desirable to use an implicit scheme. In this study, we use the Gau\ss--Legendre scheme \cite{Ise2008} to discretize our IR model.

The fact that the states also take a role as dynamic variables increases the computational effort as the number of data points increases. There are a variety of methods for solving large-scale optimization problems. One of the most widely used methods for dealing with the direct method after optimal control problems is sequential quadratic programming (SQP) \cite{Spe1998,BM2000,MMM2000}. It originates from the idea of finding the zero of the gradient of the Lagrangian function, which satisfies the complementary slacknesses as well as the predefined equality constraints. Mimicking the Hessian of the Lagrangian function using a positive-definite matrix, the method requires a quadratic programming (QP) problem to be solved using the matrix on every iterate and updating the matrix using quasi-Newton strategies to solve a similar QP problem on the next iterate. With the so-called Broyden--Fletcher--Goldfarb--Shanno (BFGS) update for the positive-definite matrix, SQP allows superlinear convergence to a local optimal solution, providing that the starting point is close to it \cite{Che1996}.

\begin{figure}[htbp!]
\centering
\subfigure[]{
\includegraphics[width=8cm,height=3cm]{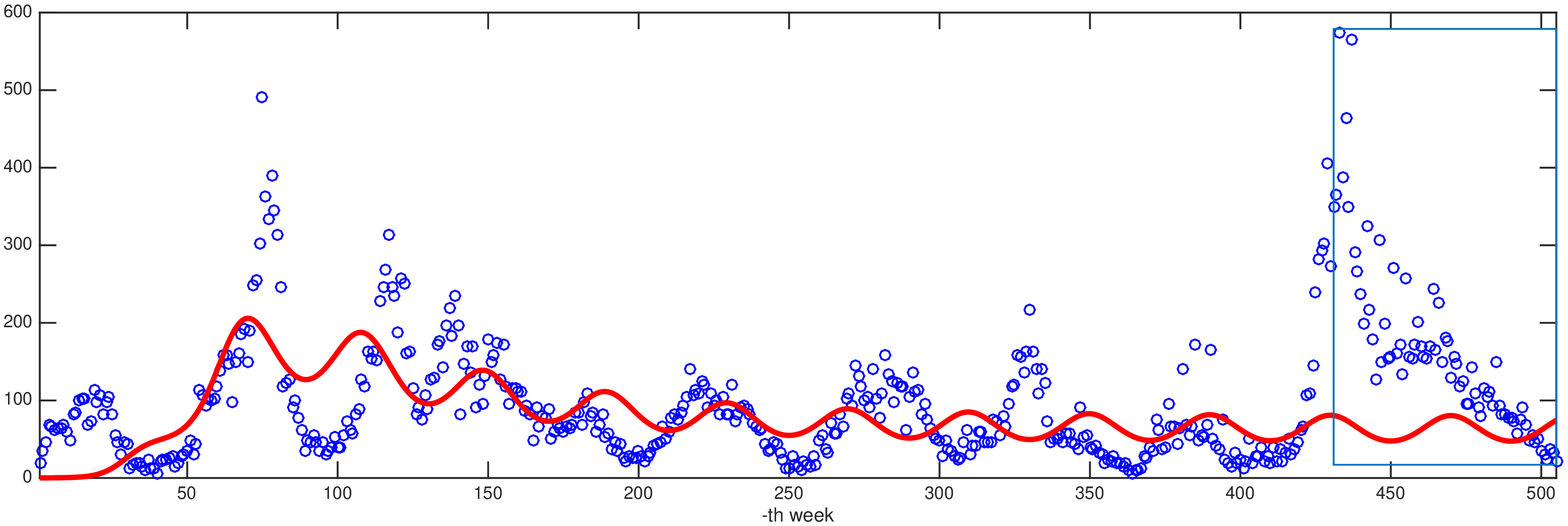}
}
\hfill\hspace*{-0.6cm}
\subfigure[]{
\includegraphics[width=8cm,height=3cm]{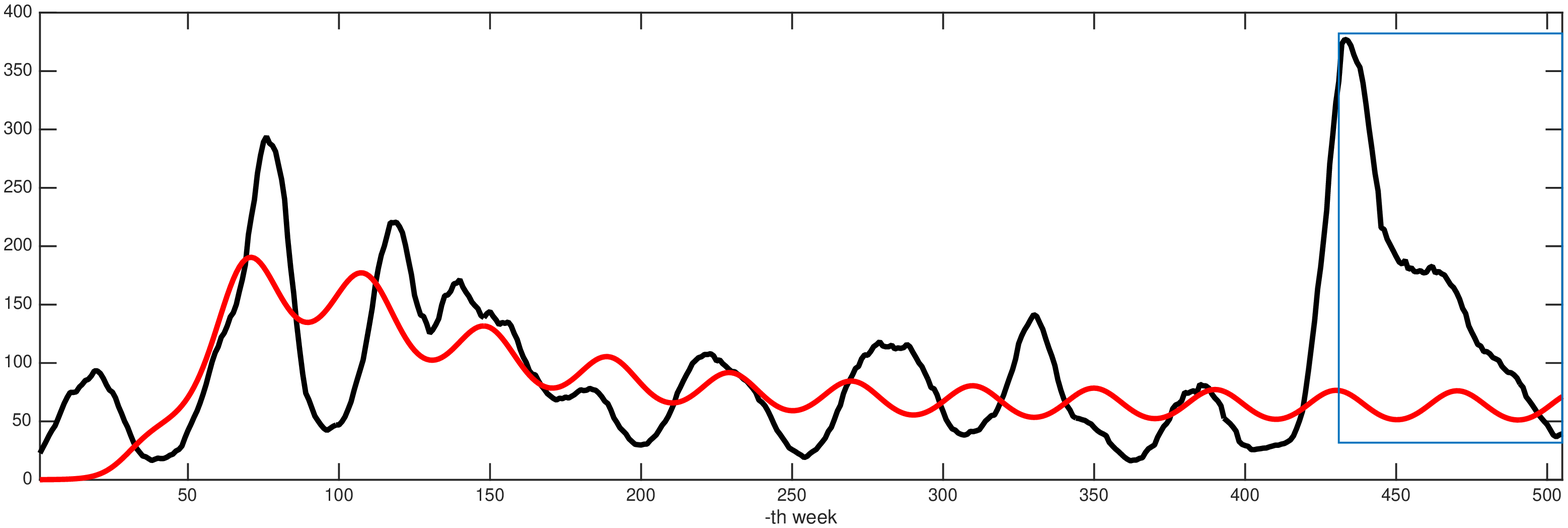}
}
\hfill
\subfigure[]{
\includegraphics[width=8cm,height=3cm]{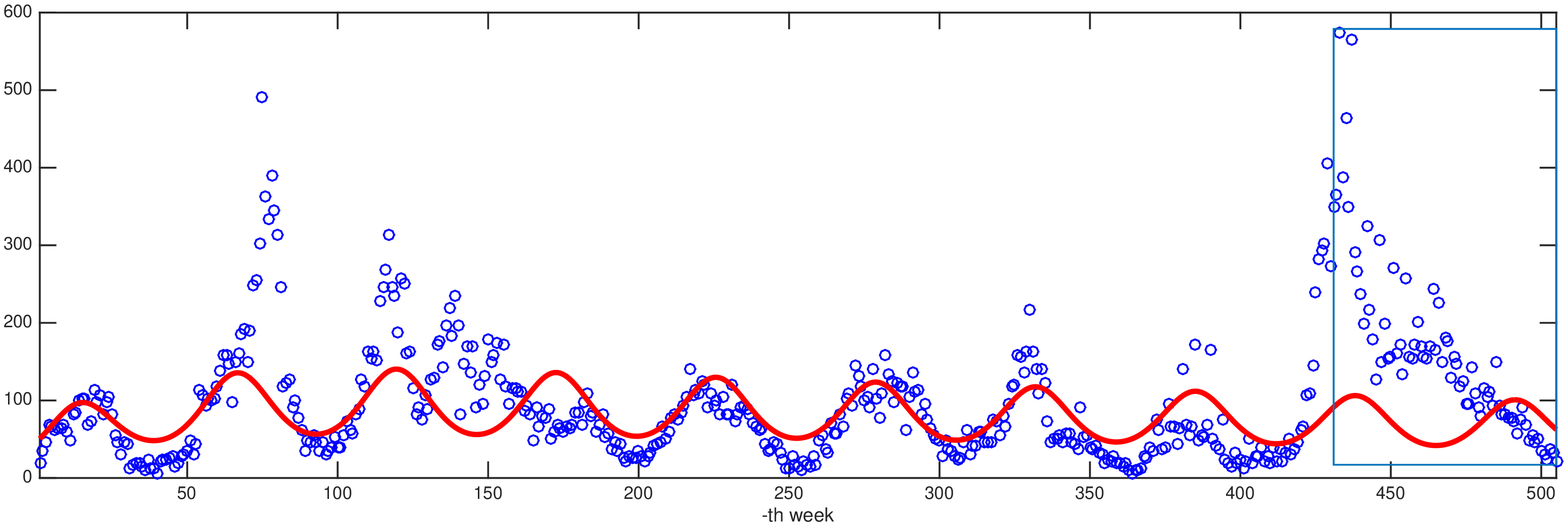}
}
\hfill\hspace*{-0.6cm}
\subfigure[]{
\includegraphics[width=8cm,height=3cm]{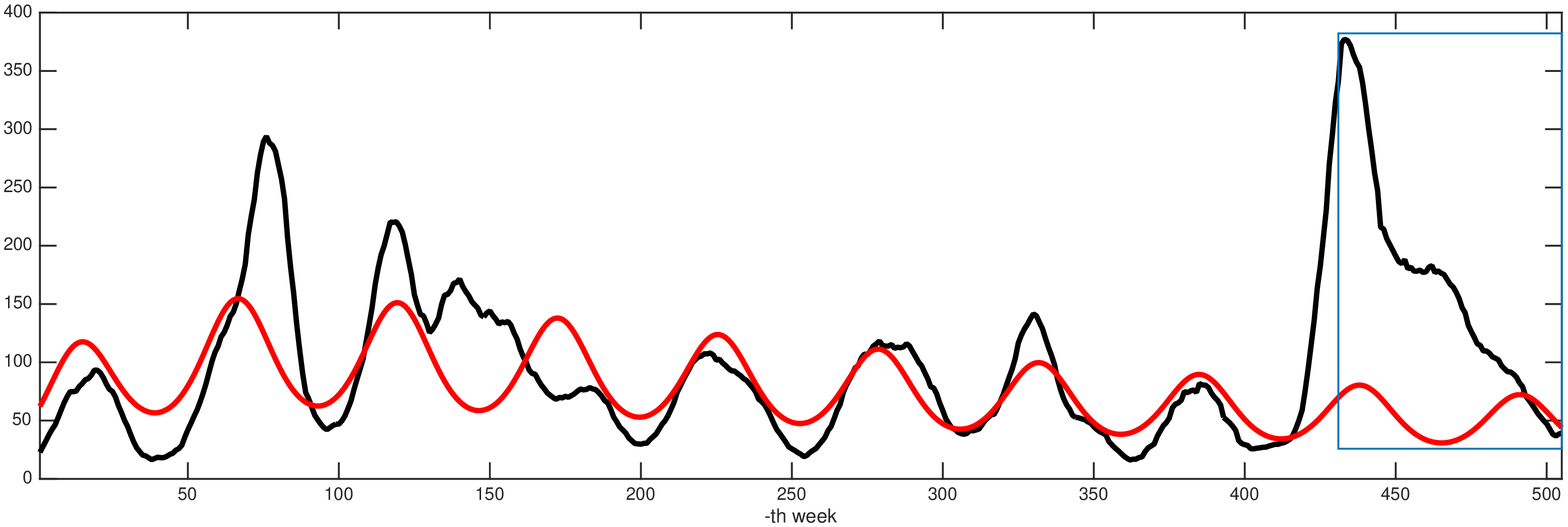}
}
\hfill
\subfigure[]{
\includegraphics[width=8cm,height=3cm]{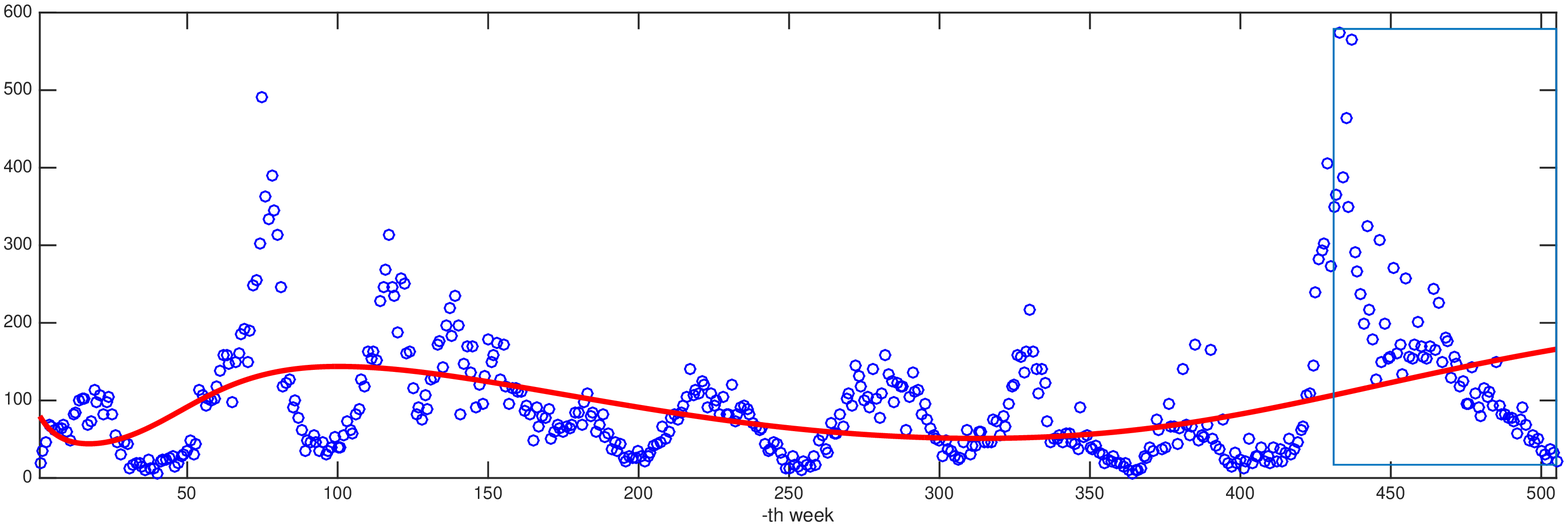}
}
\hfill\hspace*{-0.6cm}
\subfigure[]{
\includegraphics[width=8cm,height=3cm]{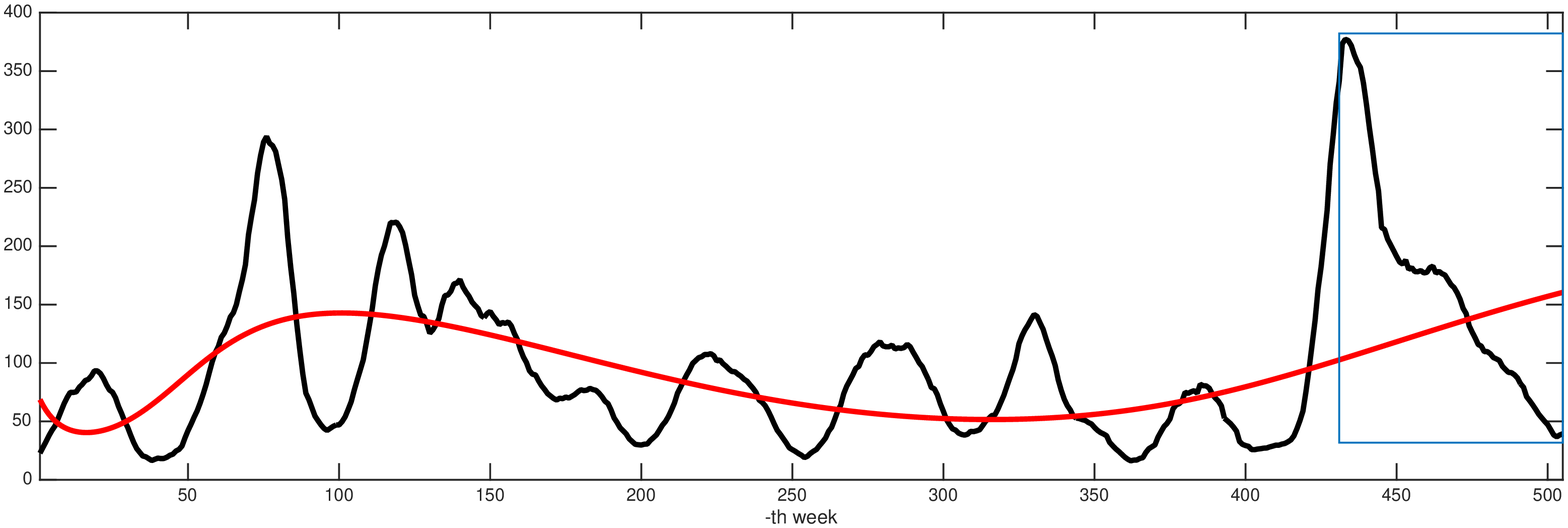}
}
\hfill
\subfigure[]{
\includegraphics[width=8cm,height=3cm]{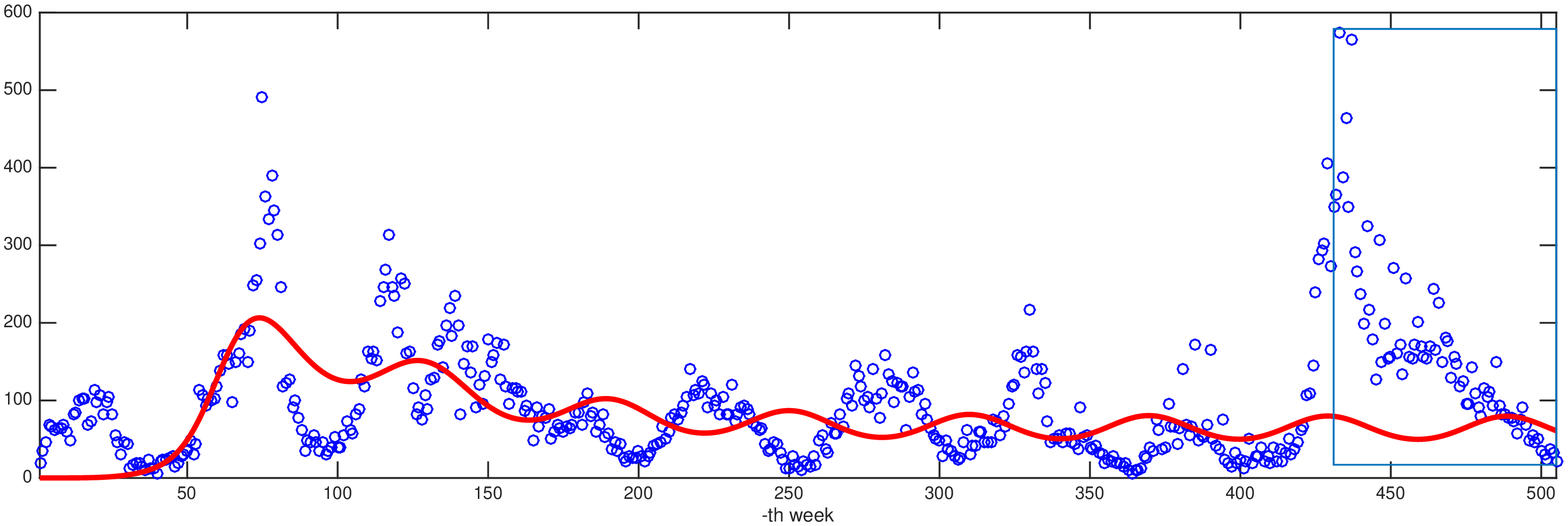}
}
\hfill\hspace*{-0.6cm}
\subfigure[]{
\includegraphics[width=8cm,height=3cm]{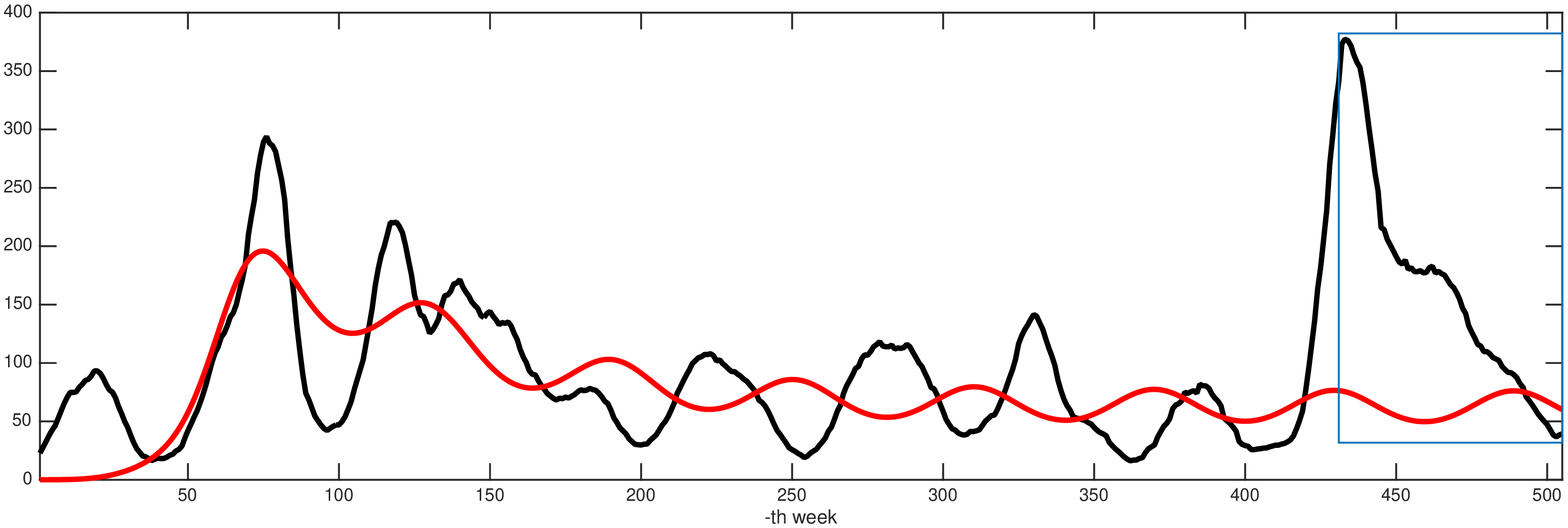}
}
\caption{Data assimilation results using the raw and filtered data without positivity of the recovered compartment for \ref{word:prob1} (a, b), \ref{word:prob2} (c, d), \ref{word:prob3} (e, f), and \ref{word:prob4} (g, h), respectively.}
\label{fig:assimil}
\end{figure}

\begin{figure}[htbp!]
\centering
\subfigure[]{
\includegraphics[width=8cm,height=3cm]{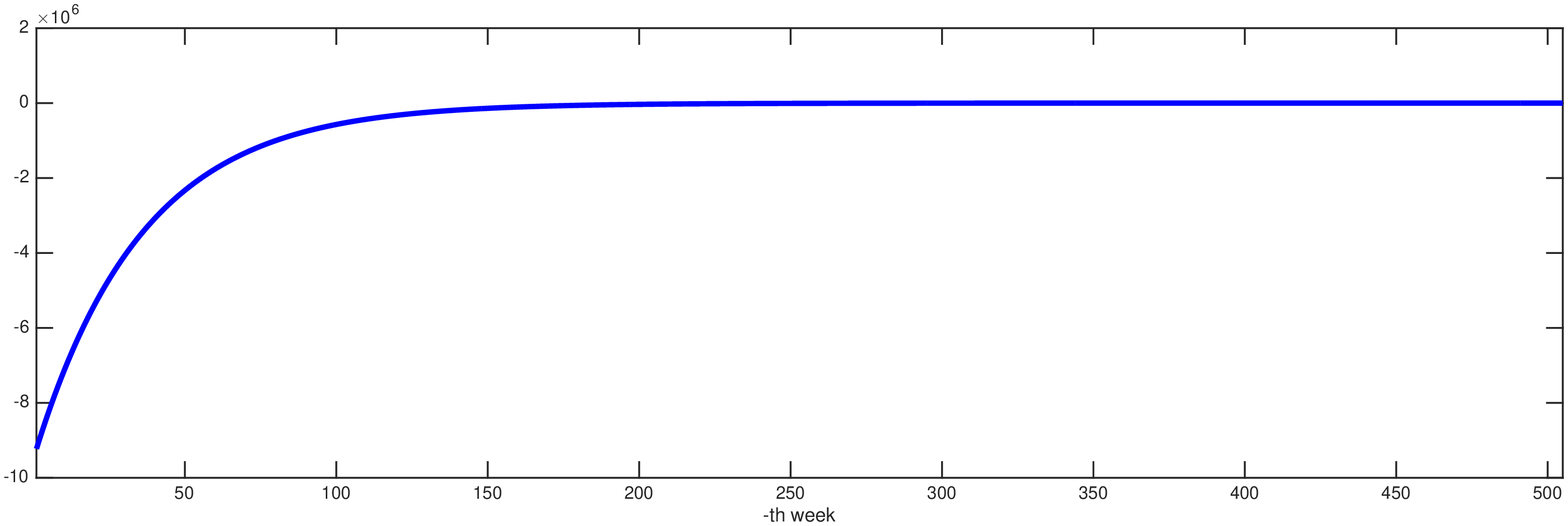}
}
\hfill\hspace*{-0.6cm}
\subfigure[]{
\includegraphics[width=8cm,height=3cm]{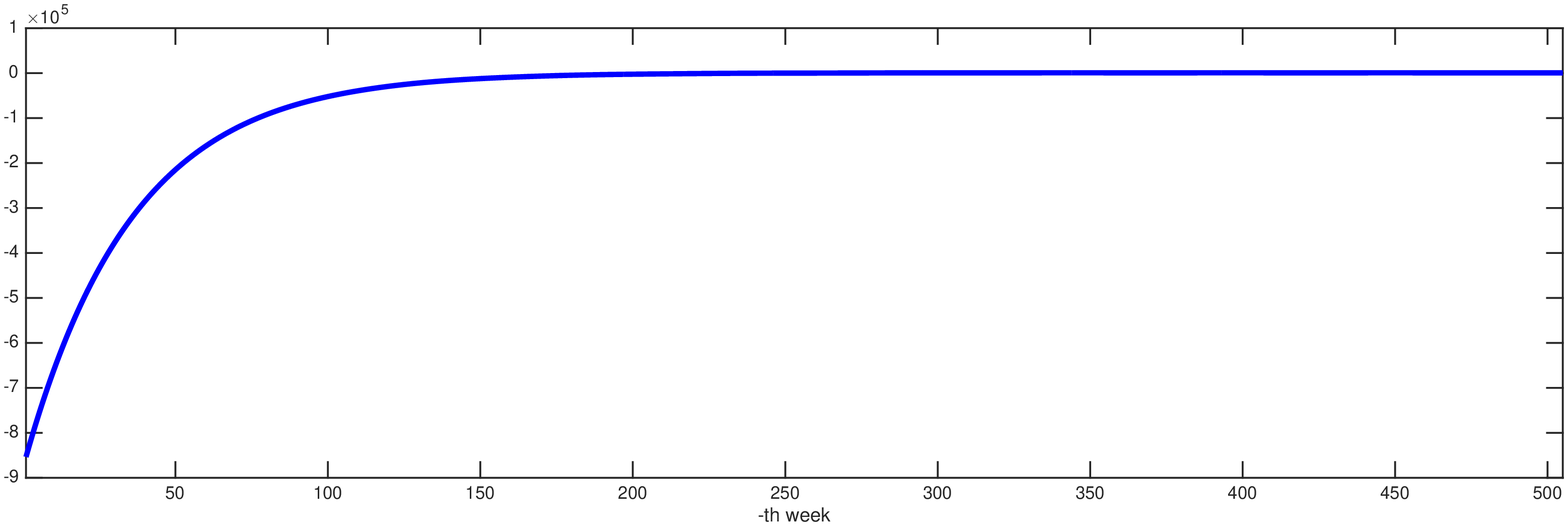}
}
\hfill
\subfigure[]{
\includegraphics[width=8cm,height=3cm]{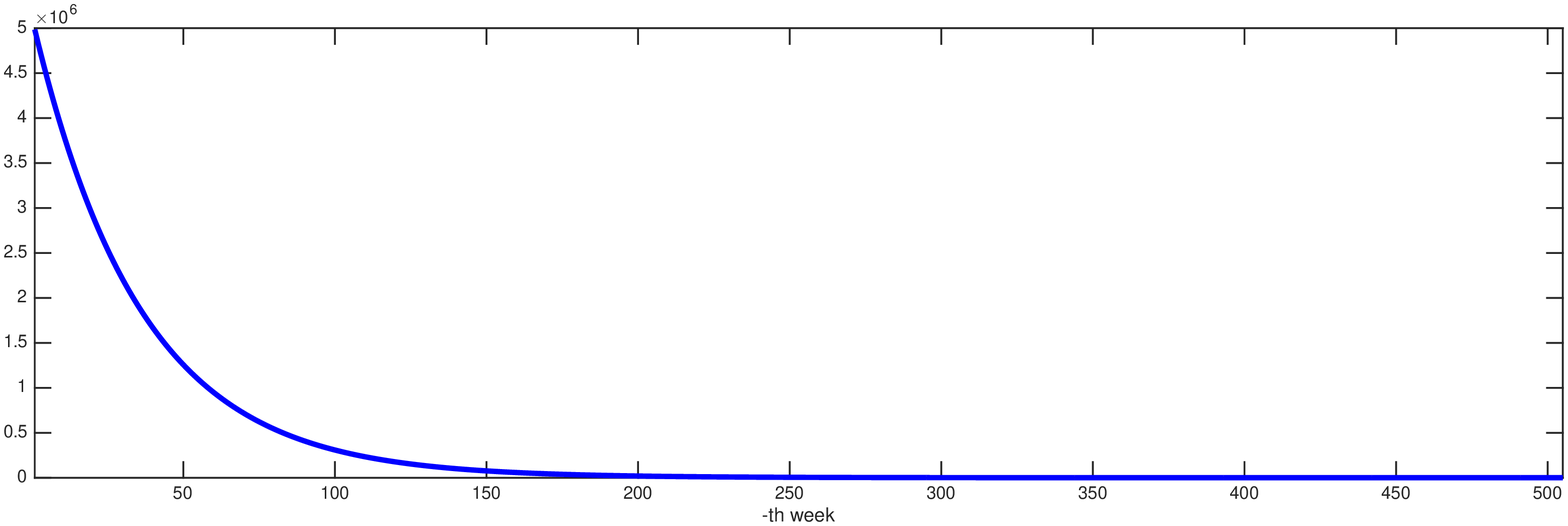}
}
\hfill\hspace*{-0.6cm}
\subfigure[]{
\includegraphics[width=8cm,height=3cm]{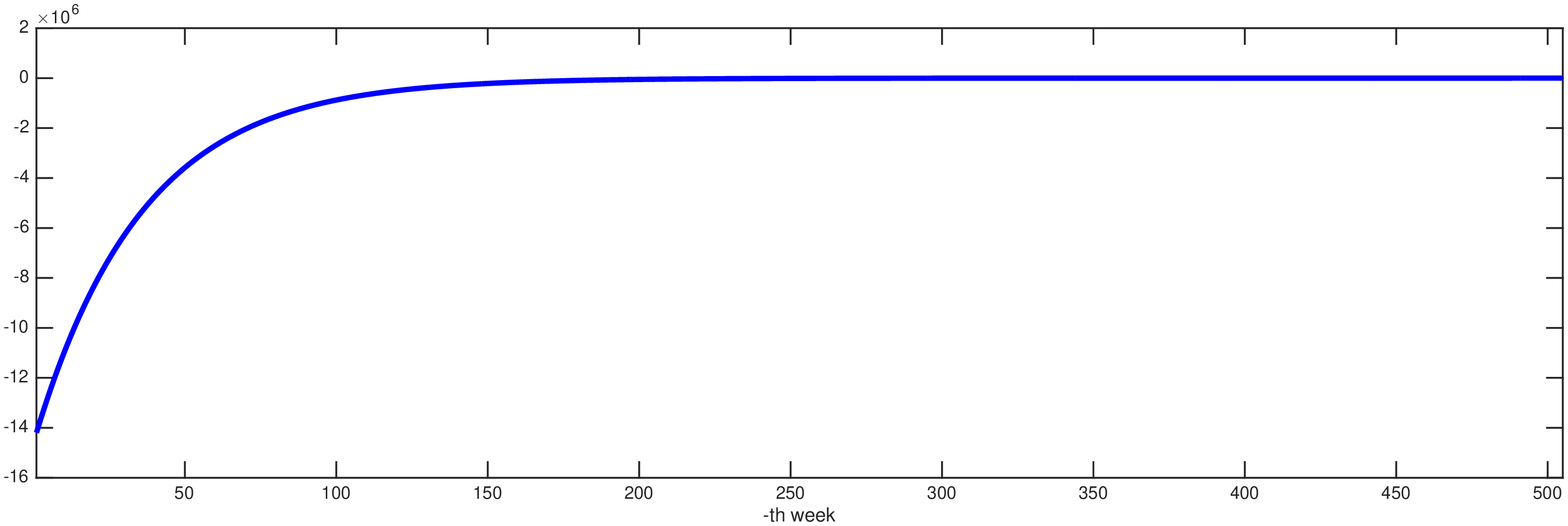}
}
\caption{Trajectories of the recovered compartments without positivity constraint for \ref{word:prob1}--\ref{word:prob4} (a--d).}
\label{fig:Rnegative}
\end{figure}

As a first strategy, we did not include the positivity of the recovered compartment as a constraint, and let the schemes discover the positivity themselves. All the parameters used in the computations are given as in Table~\ref{tab:param}. One important issue in the protocol of SQP is finding a suitable starting point. In our computation, we generated a pool of 50 distinct initial points for every scheme to avoid finding local optimal solutions. At the end of all 400 computations, we cross-checked all the schemes by using the optimal solution of one scheme as the initial point of the other schemes and vice versa. At the end, we performed as many as 456 computations. The results from Fig.~\ref{fig:assimil} show that the first two major outbreaks in 2009 and 2010 were well depicted by the solution, however returning negative recovered compartments. A comparison study also showed that all the schemes have unique results. We noted that~\ref{word:prob1},~\ref{word:prob2},~\ref{word:prob3}, and~\ref{word:prob4} produced the magnitudes of the basic reproductive number $\RM_0=1.060077$, $\RM_0=0.996444$, $\RM_0=1.430692$, and $\RM_0=1.091797$, respectively. Even though the results are fairly representative of the data and show the tendency that dengue fever is always emerging (by the magnitudes of $\RM_0$), those are pointless from our analytical perspective, owing to the appearance of negative solutions; see Fig.~\ref{fig:Rnegative}.

\begin{table}[htbp!]
\centering
\begin{tabular}{cccccccc}
\toprule[1.5pt]
$T$ & $N$ & $\mu$ & $\gamma$ & $\kappa$ & $\omega^{\ast}_{\text{r}}$ & $\omega^{\ast}_{\text{f}}$ & $\epsilon$ \\ \hline
$429$ & $10^{7}$ & $\frac{1}{65\times 12\times 4}$ & $\frac{1}{4}$ & $\frac{1}{9\times 4}$ & 0.017822 & $\omega^{\ast}_{\text{r}}$ & $10^{-2}$ \\ \bottomrule[1.5pt]
\end{tabular}
\label{tab:param}
\caption{All the parameters used for \ref{word:prob1} up to \ref{word:prob4}: $\omega^{\ast}_{\text{r}},\omega^{\ast}_{\text{f}}$ denote the most significant frequencies from the raw data and filtered data, respectively.}
\end{table}

\begin{figure}[htbp!]
\centering
\subfigure[]{
\includegraphics[width=8cm,height=3cm]{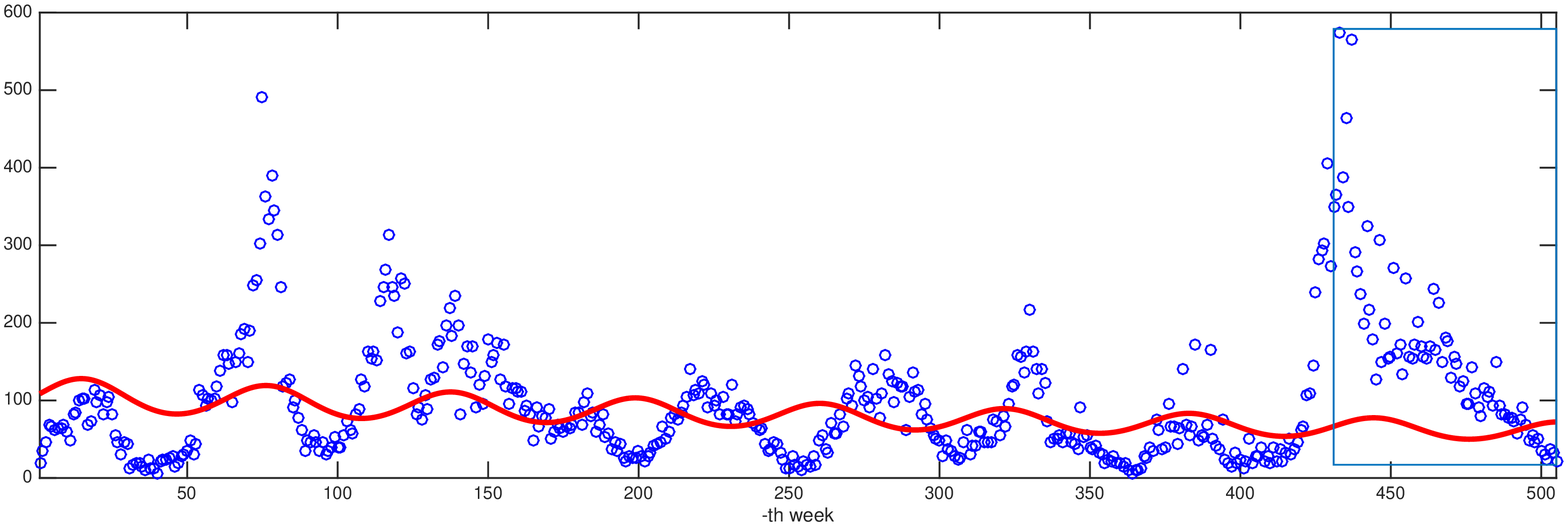}
}
\hfill\hspace*{-0.6cm}
\subfigure[]{
\includegraphics[width=8cm,height=3cm]{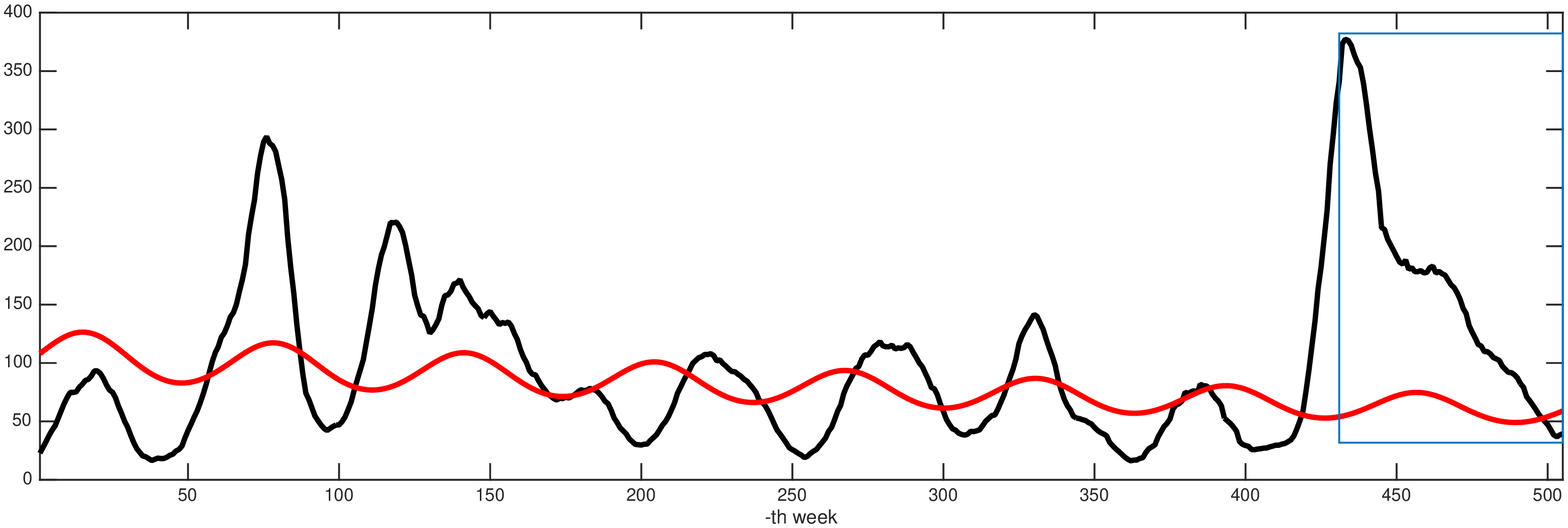}
}
\hfill
\subfigure[]{
\includegraphics[width=8cm,height=3cm]{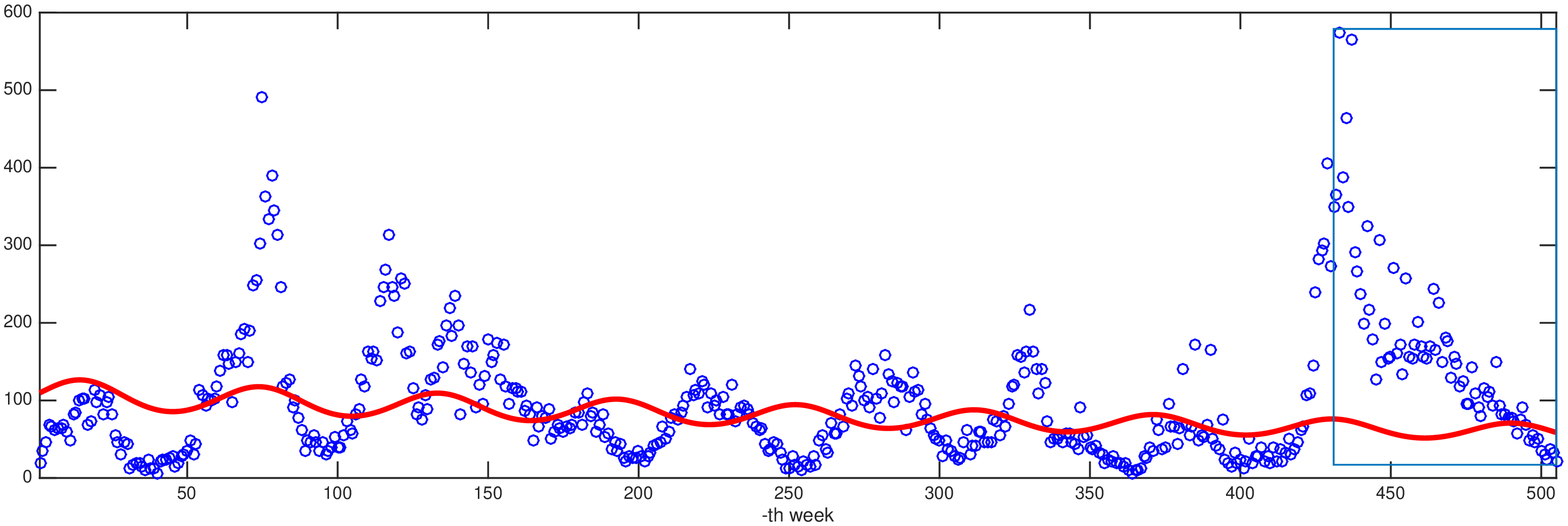}
}
\hfill\hspace*{-0.6cm}
\subfigure[]{
\includegraphics[width=8cm,height=3cm]{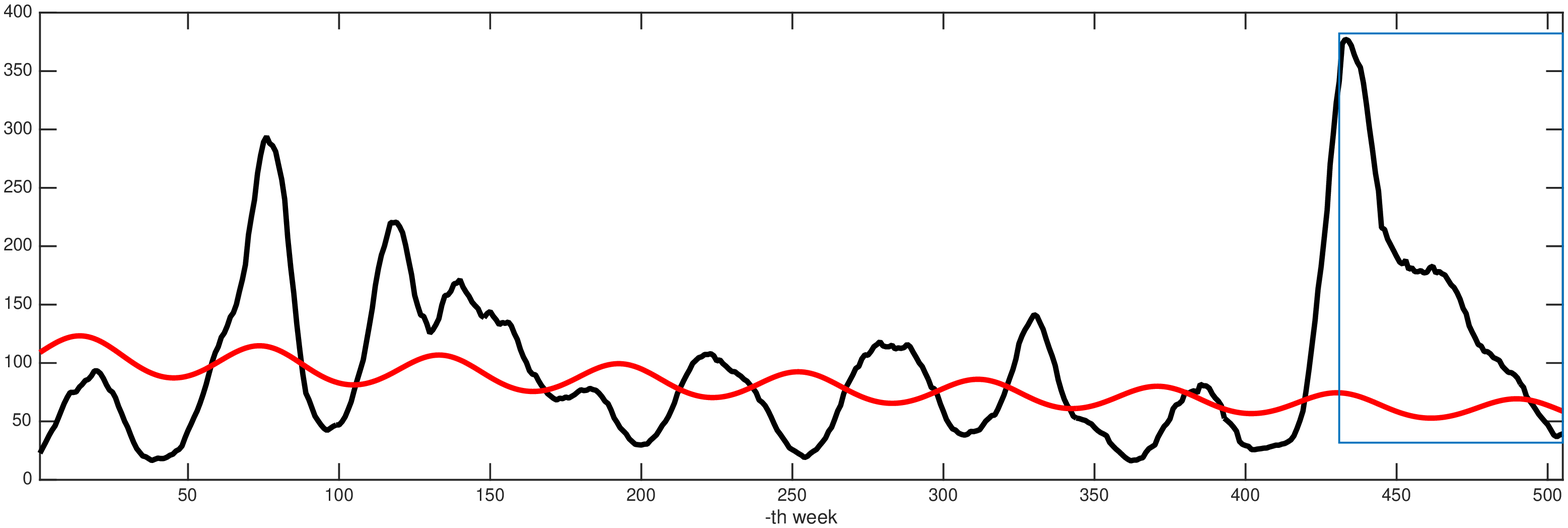}
}
\hfill
\subfigure[]{
\includegraphics[width=8cm,height=3cm]{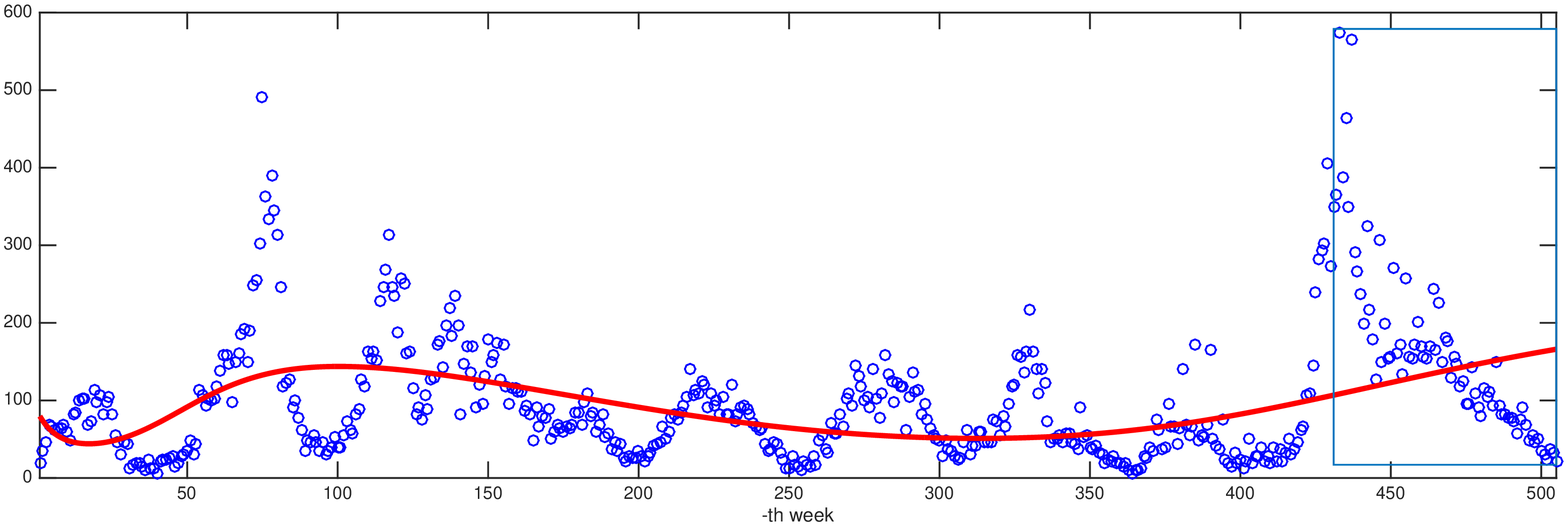}
}
\hfill\hspace*{-0.6cm}
\subfigure[]{
\includegraphics[width=8cm,height=3cm]{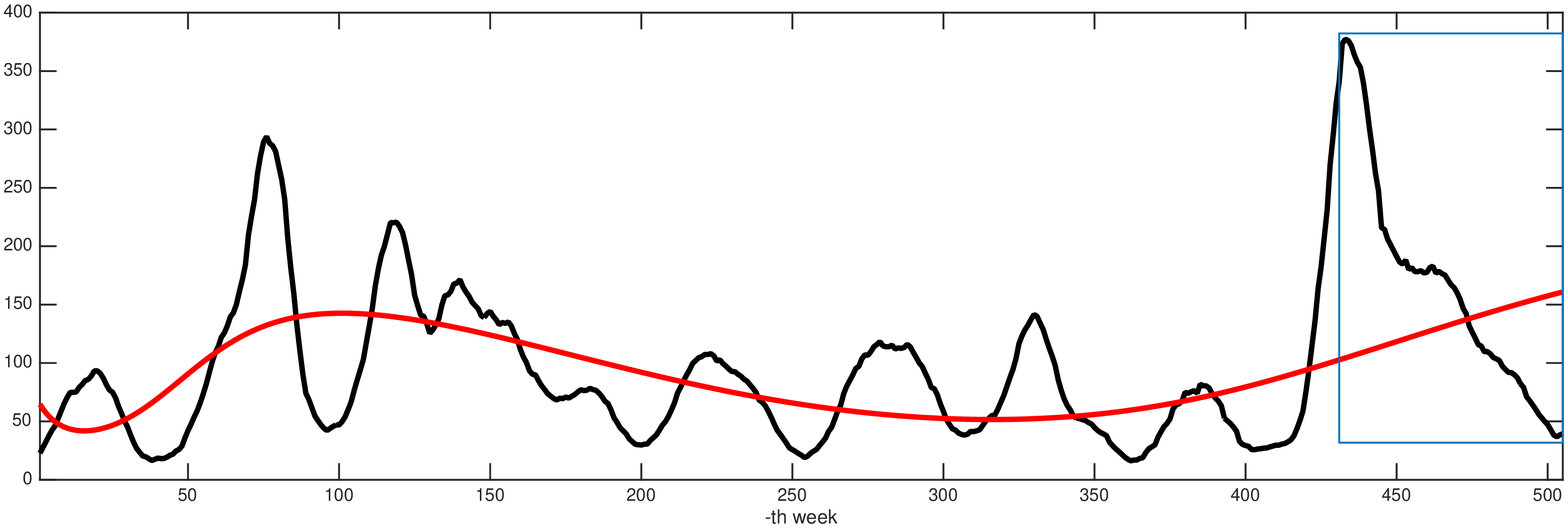}
}
\hfill
\subfigure[]{
\includegraphics[width=8cm,height=3cm]{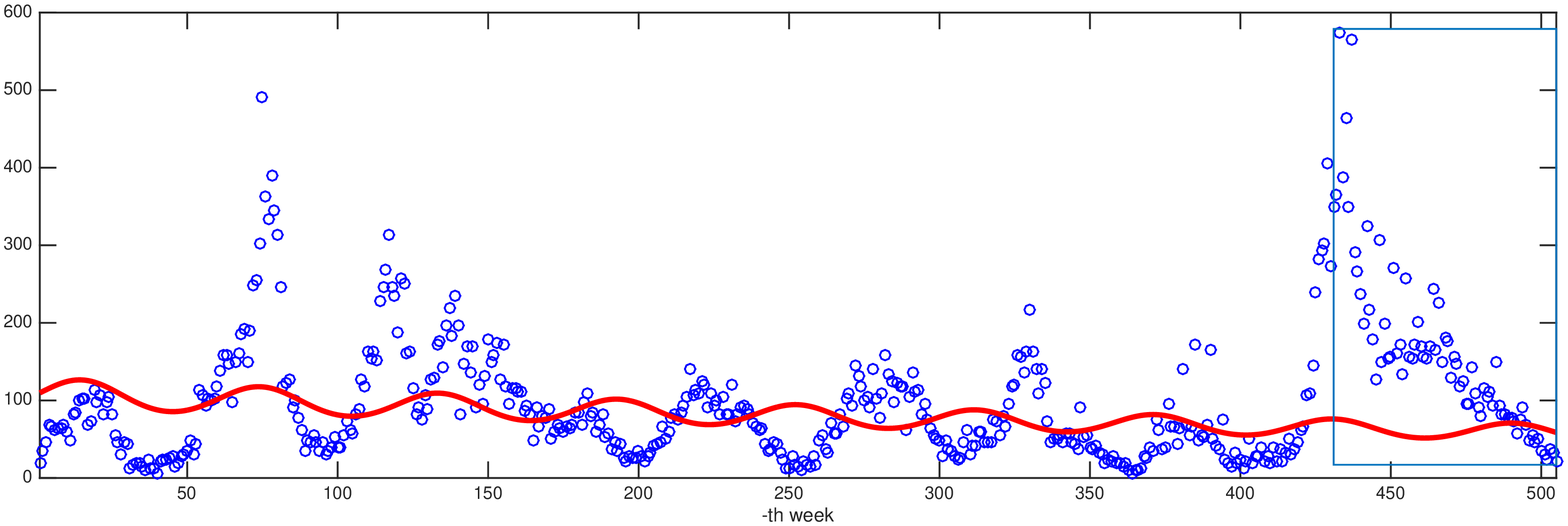}
}
\hfill\hspace*{-0.6cm}
\subfigure[]{
\includegraphics[width=8cm,height=3cm]{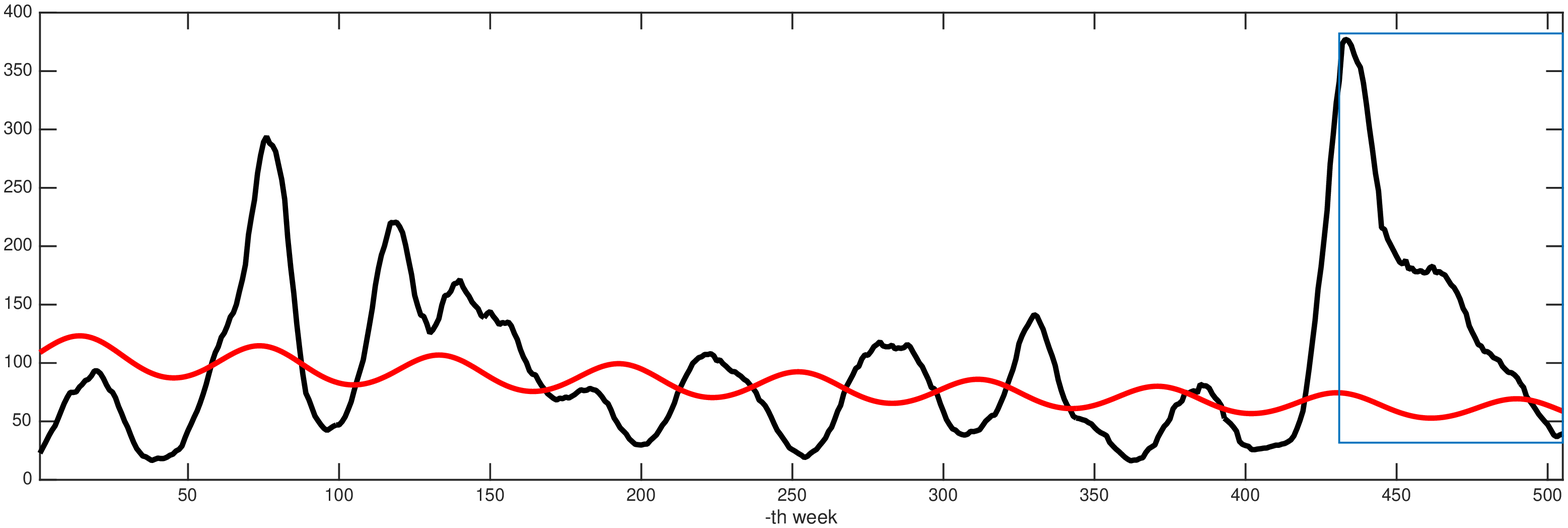}
}
\caption{Data assimilation results using the raw and filtered data with positivity of the recovered compartment for \ref{word:prob1} (a, b), \ref{word:prob2} (c, d), \ref{word:prob3} (e, f), and \ref{word:prob4} (g, h), respectively.}
\label{fig:assimilR}
\end{figure}

\begin{figure}[htbp!]
\centering
\subfigure[]{
\includegraphics[width=8cm,height=3cm]{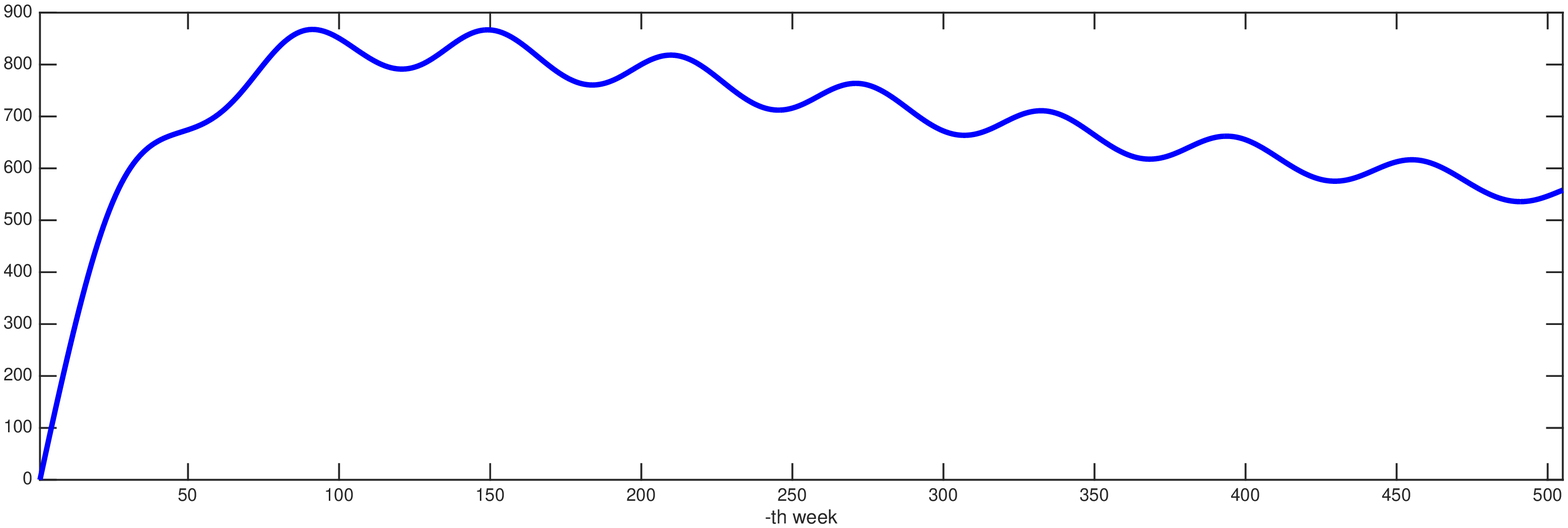}
}
\hfill\hspace*{-0.6cm}
\subfigure[]{
\includegraphics[width=8cm,height=3cm]{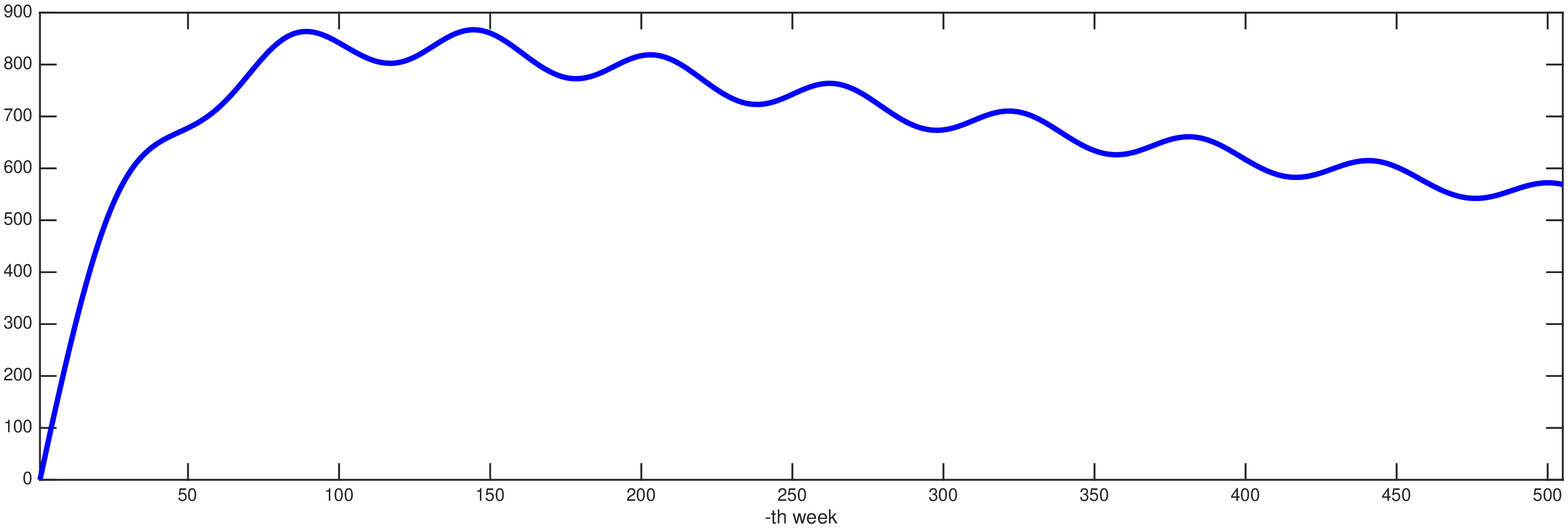}
}
\hfill
\subfigure[]{
\includegraphics[width=8cm,height=3cm]{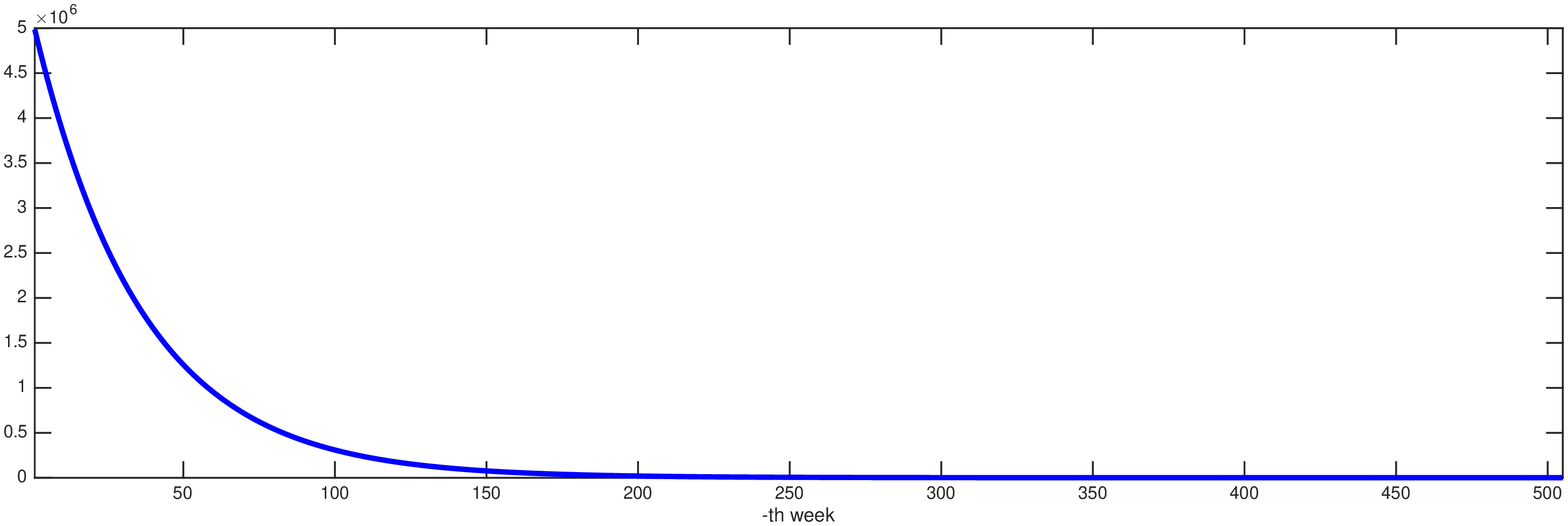}
}
\hfill\hspace*{-0.6cm}
\subfigure[]{
\includegraphics[width=8cm,height=3cm]{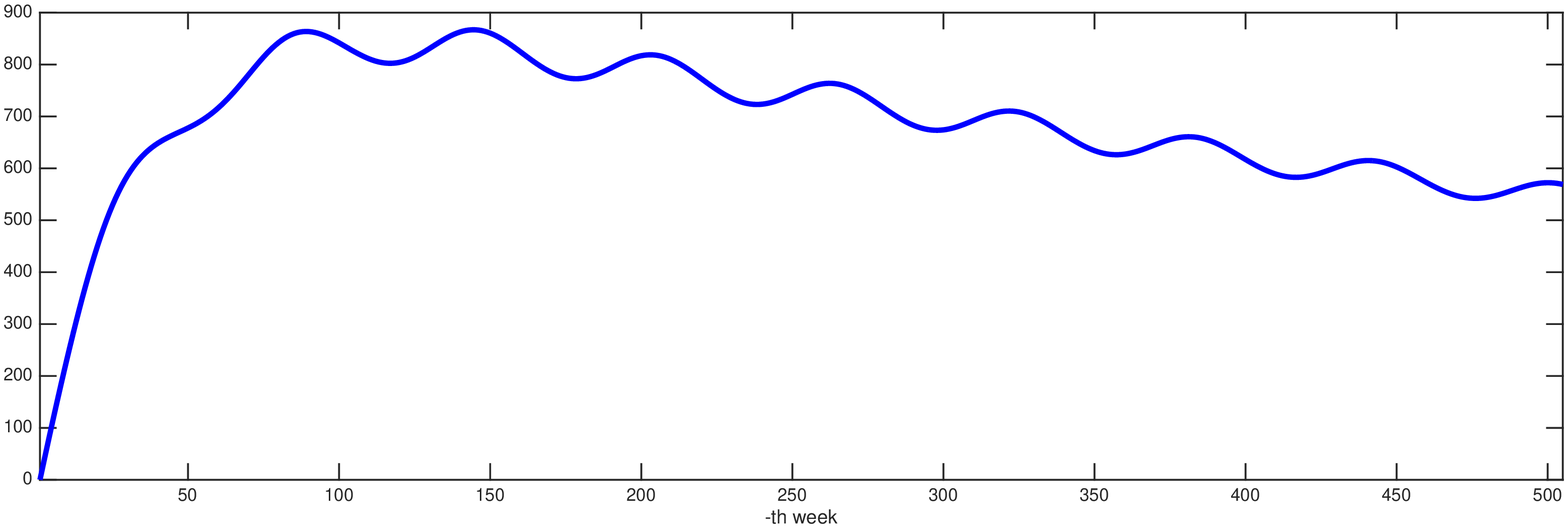}
}
\caption{Trajectories of the recovered compartments with positivity constraint for~\ref{word:prob1}--\ref{word:prob4} (a--d).}
\label{fig:Rpositive}
\end{figure}

In the second strategy, we included the positivity of the recovered compartment as another constraint in all the data assimilation schemes. Note that this constraint can simply be represented by the positivity of the initial condition of $R$, i.e.,\ $R(0)\geq 0$. The results are shown in Fig.~\ref{fig:assimilR}. The figure shows that except for~\ref{word:prob3}, all schemes are more robust as compared with the first strategy in the sense that they produce similar behavior of the optimal solution trajectories. We note that removing noise does not have any significant effect on the final result for each scheme. This is explained according to Table~\ref{tab:result} that all the numerical values of the basic reproductive number $\RM_0$ merely are the same.

\begin{table}[htbp!]
\centering
\resizebox{1\textwidth}{!}{
\begin{tabular}{c|c|c|c|c}
\toprule[1.5pt]
Par., Perf. & \ref{word:prob1} & \ref{word:prob2} & \ref{word:prob3} & \ref{word:prob4} \\ \hline
Raw data & $\left(\begin{array}{c}
1.75\times 10^{2}\\
14.6533\\
0.01630\\
7.01\times 10^{2}\\
1.09 \times 10^{2}\\
6.95 \times 10^{-20}
\end{array}
\right)$, $\left(\begin{array}{c}
1.94\times 10^{6}\\
0.99549
\end{array}
\right)$& $\left(\begin{array}{c}
1.75\times 10^{2}\\
13.2732\\
0.01682\\
7.01\times 10^{2}\\
1.10\times 10^{2}\\
3.13\times 10^{-21}
\end{array}
\right)$, $\left(\begin{array}{c}
1.95\times 10^{6}\\
0.99529
\end{array}
\right)$ & $\left(\begin{array}{c}
1.02\times 10^{-5}\\
1.85\times 10^{-6}\\
1.00172\\
2.85\times 10^{-5}\\
79.6403\\
4.99\times 10^{6}
\end{array}
\right)$, $\left(\begin{array}{c}
1.68\times 10^{6}\\
1.43072\\
10^{4}
\end{array}
\right)$ & $\left(\begin{array}{c}
0.18527\\
0.01409\\
0.01682\\
0.74361\\
1.10\times 10^{2}\\
9.09\times 10^{-22}
\end{array}
\right)$, $\left(\begin{array}{c}
1.95\times 10^{6}\\
0.99530\\
10
\end{array}
\right)$ \\ \hline
Filtered data & $\left(\begin{array}{c}
1.75\times 10^{2}\\
13.5855\\
0.01586\\
7.02\times 10^{2}\\
1.08\times 10^{2}\\
6.58\times 10^{-20}
\end{array}
\right)$, $\left(\begin{array}{c}
1.31\times 10^{6}\\
0.99539
\end{array}
\right)$ & $\left(\begin{array}{c}
1.75\times 10^{2}\\
11.5947\\
0.01682\\
7.01\times 10^{2}\\
1.09\times 10^{2}\\
2.06\times 10^{-20}
\end{array}
\right)$, $\left(\begin{array}{c}
1.33\times 10^{6}\\
0.99533
\end{array}
\right)$ & $\left(\begin{array}{c}
1.06\times 10^{-5}\\
1.81\times 10^{-6}\\
1.00171\\
3.01\times10^{-5}\\
64.6548\\
4.61\times 10^{6}
\end{array}
\right)$, $\left(\begin{array}{c}
1.04\times 10^{6}\\
1.40370\\
10^{4}
\end{array}
\right)$ & $\left(\begin{array}{c}
0.20337\\
0.01350\\
0.01682\\
0.81623\\
1.09\times 10^{2}\\
4.49\times 10^{-19}
\end{array}
\right)$, $\left(\begin{array}{c}
1.33\times 10^{6}\\
0.99534\\
10
\end{array}
\right)$ \\ 
\bottomrule[1.5pt]
\end{tabular}
}
\label{tab:result}
\caption{Optimal parameters $\text{Par.}:=(\alpha,
\delta,
\omega,
\nu,
I_0,
R_0)$ and performance measures $\text{Perf.}:=(\lVert I-\Id\rVert^2_2,\RM_0)$ and $\text{Perf.}:=(\lVert I-\Id\rVert^2_2,\RM_0,\lambda)$ for~\ref{word:prob3} and \ref{word:prob4}, respectively.}
\end{table}

Vanishing $\lambda$ apparently assists us as we avoid having to choose suitable values for it, for the final results are highly sensitive to changes in $\lambda$. This argues that the results are more sensitive to $\lambda$ than the format of the data. Moreover, all the schemes except~\ref{word:prob3} are considered benchmarks in our computations as they are able to pronounce the fluctuations of the data. Apparently, no scheme adequately depicts the high amplitude of the test (predicted) data in the last 32 weeks. Therefore, what remains problematic is if the result $\RM_0<1$ according to Table~\ref{tab:result} is reliable. According to Theorem~\ref{thm:stableIR} and~\ref{item:trivial}, we immediately conclude that the trajectory of the disease incidence will (slowly) tend to zero in the long run. Another issue is that even though all the schemes except~\ref{word:prob3} give some fluctuations, they produce higher extrinsic errors than \ref{word:prob3}, which is why some improvement is needed.

\section{Discussion and outlook}
The numerical framework that we dealt with from the previous section brings a direct connection to the analytical results from Floquet theory. In addition, the framework allows us to have fewer unknown parameters and variables as compared with a framework in which some parameters are completely time varying, thus inducing less computational effort. However, a first disadvantage appears from a brief observation over the numerical results that no single solution from the eight schemes matched well with the given data. This happens since the data were forced to have exactly one significant frequency, which is a stringent condition. One might see incidence data that have only one significant frequency, but a distinct seasonal incidence pattern requires the nature of the solution to have more frequencies and amplitudes, because then the density and estrangement between two adjacent outbreaks can be more accurately depicted. To mitigate this issue, one may use the following model
\begin{equation*}
\beta(t) = \alpha + \sum_{j=1}^m \delta_j\cos(2\pi\omega_{j}t) \quad \text{where }\sigma_j:=1\slash\omega_j,
\end{equation*}
and $m$ denotes the desired number of significant frequencies from the data that are extracted from the Fourier transformation. Unless every quotient $\sigma_j\slash\sigma_k$ returns a rational number, $\beta$ cannot be a periodic function, because then Floquet theory no longer applies. A typical example would be when the Fourier transformation produces $m$ rationally independent significant periods, e.g.,\ with just one period being irrational. Apparently, in any numerical implementation, the Fourier transformation produces periods that are of decimal number with a finite number of digits after the decimal point. This simply keeps us on track. Nonetheless, one idea when an irrational period appears can be to slightly ``shift'' it around its neighborhood for which all the desired periods share rational dependencies. We consider the appearing irrational period as an open research problem. For now, we use the following three frequencies from the data
\begin{equation*}
\omega^{\ast}_{1}=0.005941, \quad \omega^{\ast}_{2}=0.017822, \quad 
\omega^{\ast}_{3}=0.0237624,\\
\end{equation*}
to obtain solutions as described in Fig.~\ref{fig:assimilmultR} with the numerical results~\eqref{eq:par}--\eqref{eq:perf}. Note that we require the additional constraint $\alpha\geq \sum_{j=1}^m|\delta_j|$ to keep $\beta$ nonnegative. From the figure, we clearly see that the historical outbreaks are more naturally depicted by the solutions. The scheme was also able to suggest an outbreak within the test time window, which happens to depict the actual outbreak (even though it appears with an underestimation). This is at least in contrast to the previous strategies (Fig.~\ref{fig:assimil}--\ref{fig:assimilR}). The numerical results also suggest that the basic reproductive number $\RM_0$ is slightly greater than one, whereas $\RM_{\max}$ is slightly less than one. Accordingly, there exists a nontrivial periodic solution that is locally asymptotically stable. This analytical finding is then supported by the numerical realization of the system solution on an extended time domain as in Fig.~\ref{fig:solmult}:

\begin{figure}[htbp!]
\centering
\subfigure[]{
\includegraphics[width=8cm,height=3cm]{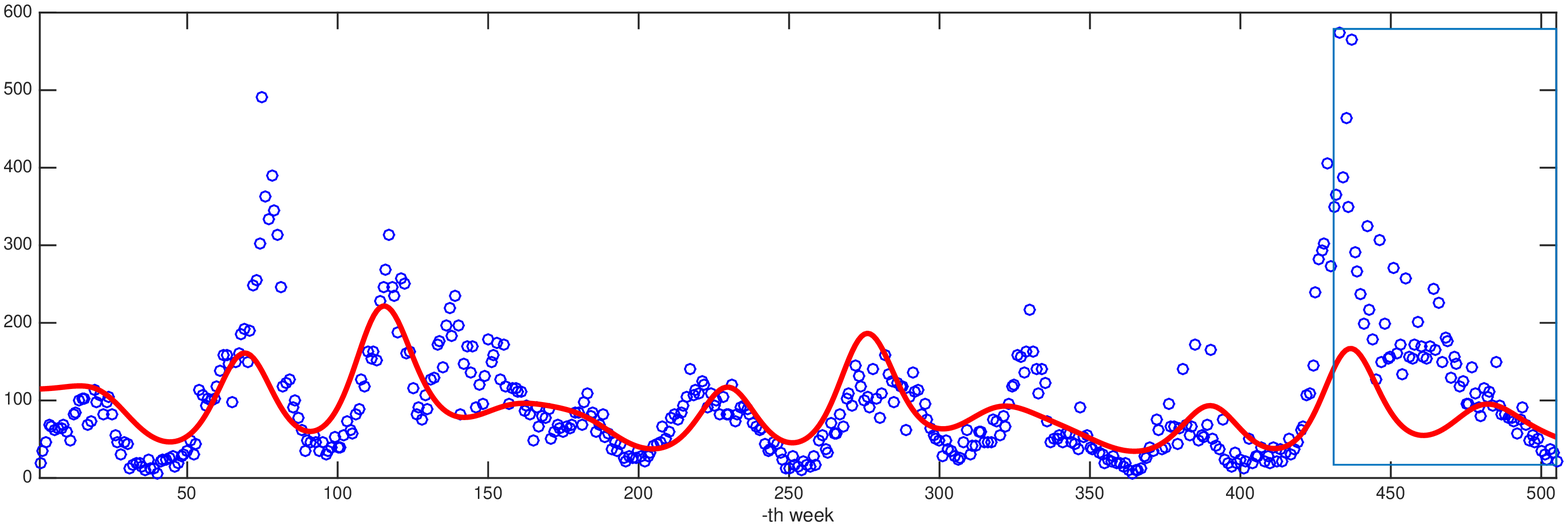}
}
\hfill\hspace*{-0.6cm}
\subfigure[]{
\includegraphics[width=8cm,height=3cm]{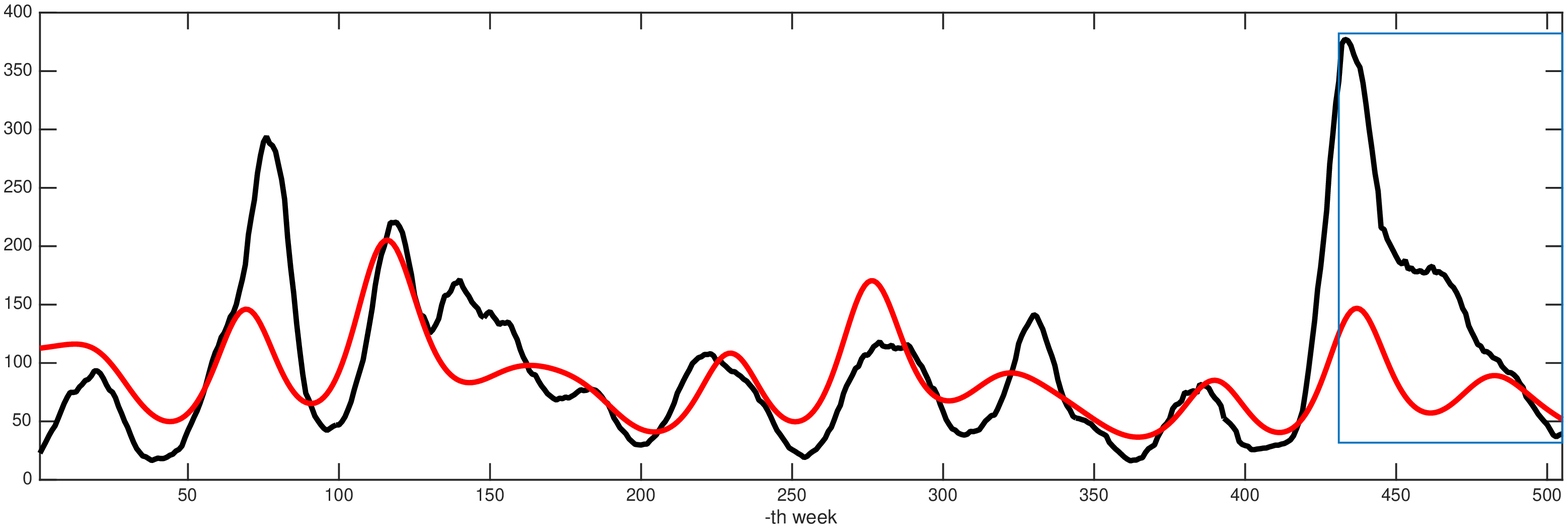}
}
\hfill
\subfigure[]{
\includegraphics[width=8cm,height=3cm]{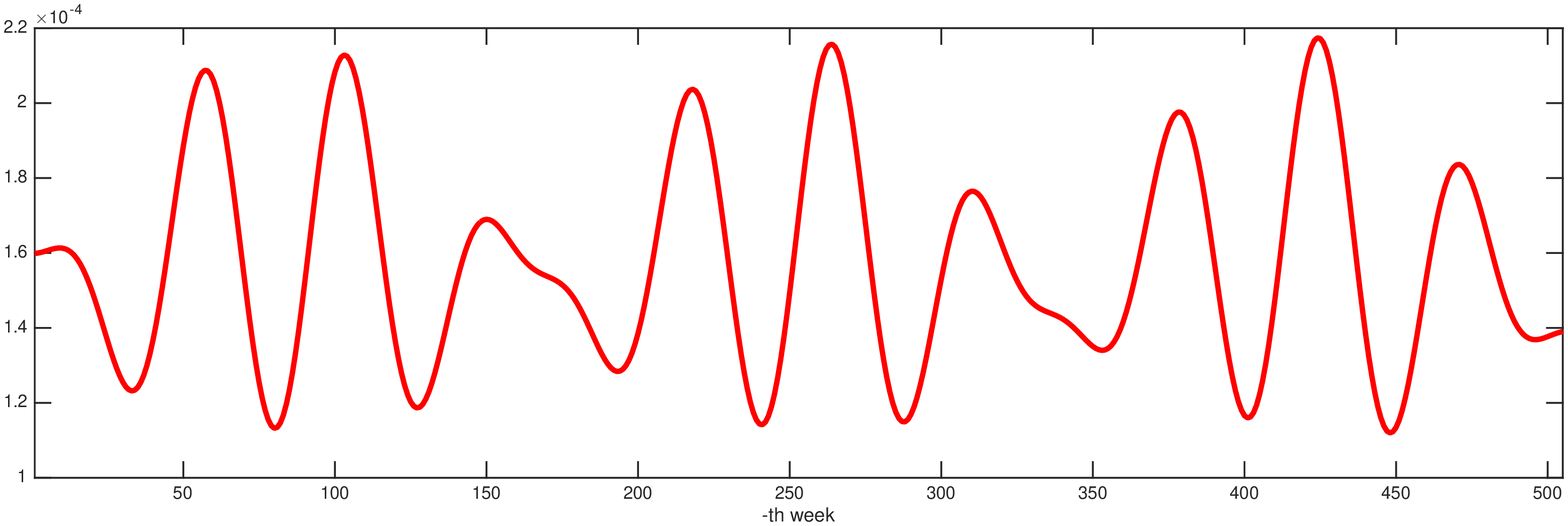}
}
\hfill\hspace*{-0.6cm}
\subfigure[]{
\includegraphics[width=8cm,height=3cm]{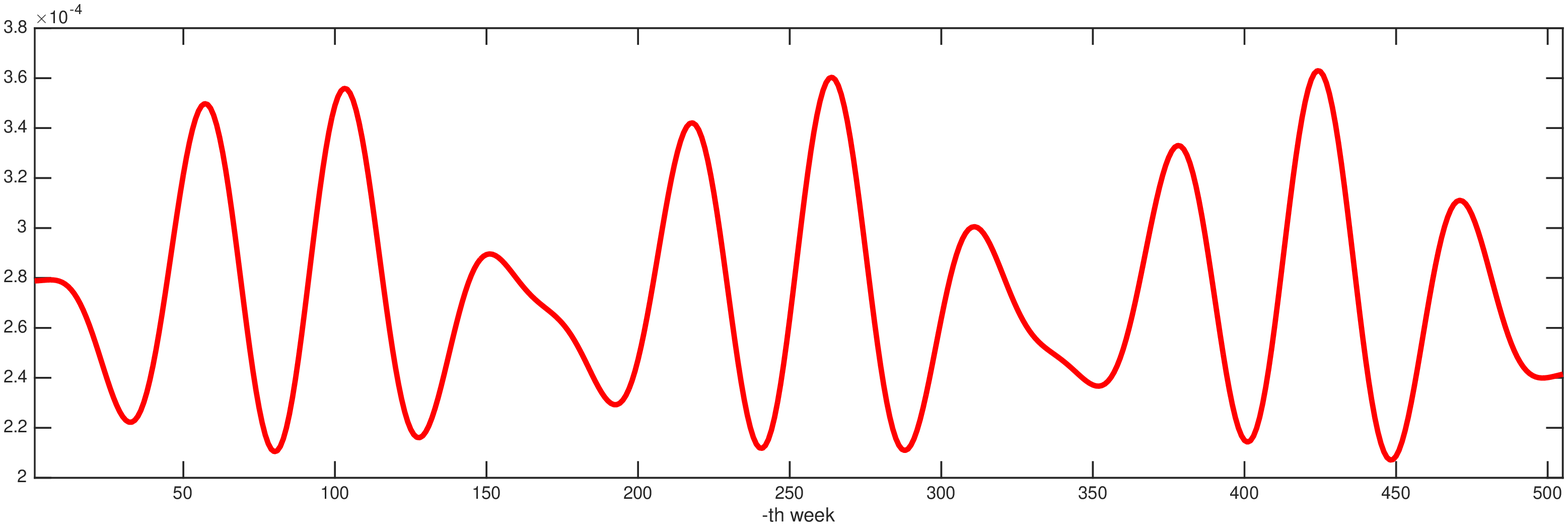}
}
\hfill
\subfigure[]{
\includegraphics[width=8cm,height=3cm]{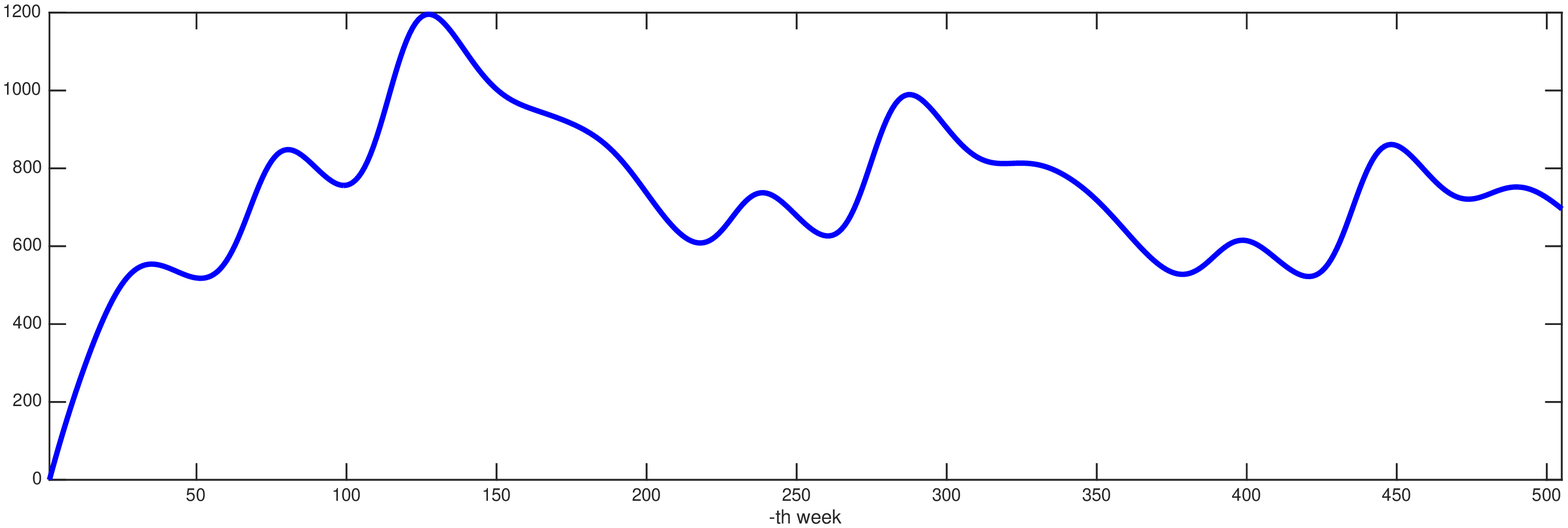}
}
\hfill\hspace*{-0.6cm}
\subfigure[]{
\includegraphics[width=8cm,height=3cm]{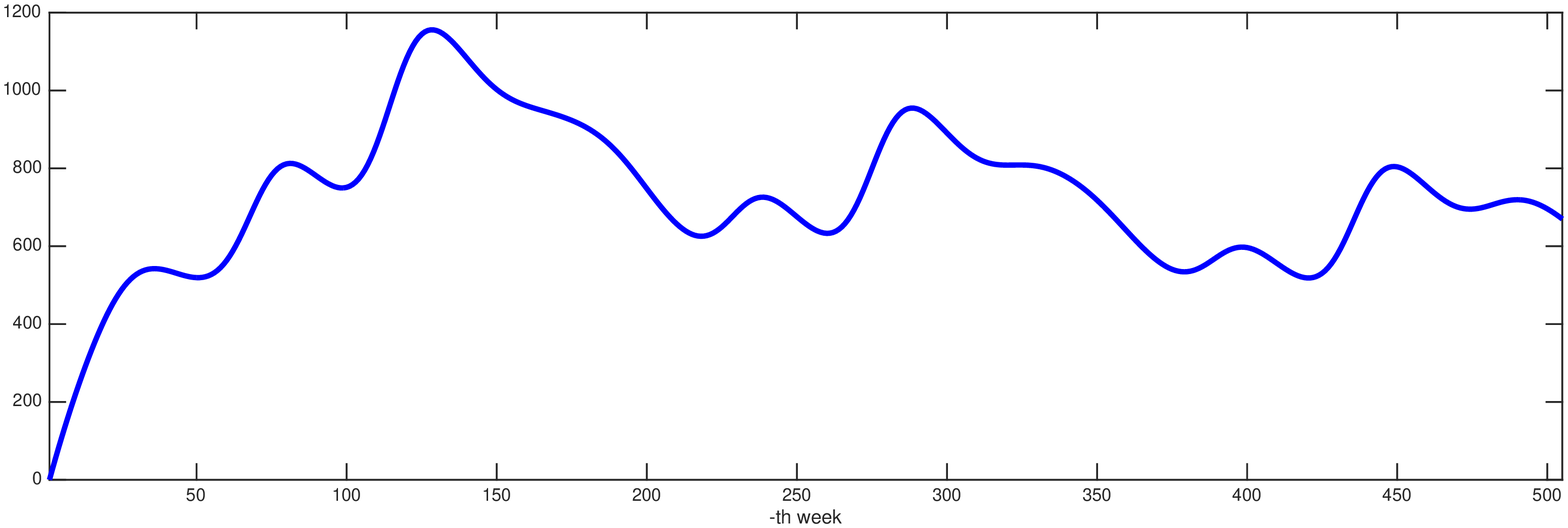}
}
\caption{Data assimilation results using the raw (a,c,e) and filtered data (b,d,f) with the positivity of the recovered compartment: the first, second, and third lines depict the optimal trajectory of the infected compartment, the optimal trajectory of $\beta$, and that of the recovered compartment.}
\label{fig:assimilmultR}
\end{figure}


\begin{figure}[htbp!]
\centering
\includegraphics[width=17cm,height=4.2cm]{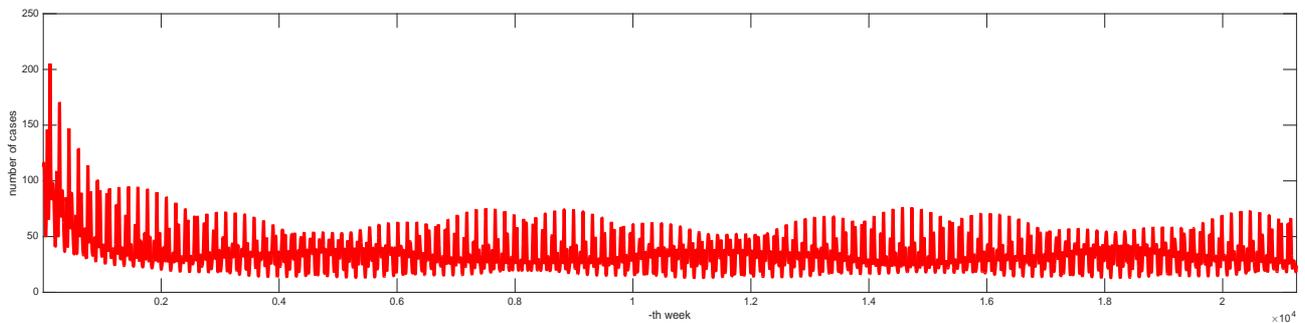}
\caption{Extension of the trajectory of the infected compartment from Fig.~\ref{fig:assimilmultR} (a).}
\label{fig:solmult}
\end{figure}

\begin{equation}\label{eq:par}
\left(\begin{array}{c}
\alpha\\
\delta_1\\
\omega_1\\
\delta_2\\
\omega_2\\
\delta_3\\
\omega_3\\
\nu\\
I_0\\
R_0\\
\end{array}
\right)_{\text{raw}}=\left(\begin{array}{c}
1.57434\times 10^{-4}\\
-8.50356\times 10^{-6}\\
0.00609\\
3.18808\times 10^{-5}\\
0.01882\\
-2.09876\times 10^{-5}\\
0.02476\\
6.23095\times 10^{-4}\\
1.14701\times 10^{2}\\
5.16297\times 10^{-21}\\
\end{array}
\right)\text{ and }
\left(\begin{array}{c}
\alpha\\
\delta_1\\
\omega_1\\
\delta_2\\
\omega_2\\
\delta_3\\
\omega_3\\
\nu\\
I_0\\
R_0\\
\end{array}
\right)_{\text{filtered}}=\left(\begin{array}{c}
2.73535\times 10^{-4}\\
-1.41929\times 10^{-5}\\
0.00609\\
4.84729\times 10^{-5}\\
0.01882\\
-2.90086\times 10^{-5}\\
0.02476\\
0.00109\\
1.12664\times 10^{2}\\
3.61013\times 10^{-20}\\
\end{array}
\right),
\end{equation}

\begin{equation}\label{eq:perf}
\resizebox{0.9\textwidth}{!} 
{ $
\left(\begin{array}{c}
\lVert I-I^{\text{DATA}}\rVert_2^2\\
\RM_0\\
\RM_{\max}
\end{array}
\right)_{\text{raw}}=\left(\begin{array}{c}
1.33050\times 10^{6}\\
1.00937\\
0.99891
\end{array}
\right)\text{ and }
\left(\begin{array}{c}
\lVert I-I^{\text{DATA}}\rVert_2^2\\
\RM_0\\
\RM_{\max}
\end{array}
\right)_{\text{filtered}}=\left(\begin{array}{c}
8.26737\times 10^{5}\\
1.00319\\
0.99307
\end{array}
\right).
$}
\end{equation}

One limitation of this work is that we did not perform a secondary analysis of the data, as they were directly put into the modeling. However, the daily epidemic data were based on the hospitalized cases only. Underreporting may occur for which some people may come to the local polyclinics (\emph{Puskesmas}), and were not referred to the hospitals due to asymptomatic infections. As far as the IR model is concerned, one should note that the model only describes the dynamics of the human population, considered to be constant over time. This assumption slightly violates the fact that the number of people living in Jakarta increases from time to time \cite{SI2016}. The advantages of using the IR model are that it induces simplicity in the numerical optimization during data assimilation, and the existence--stability results can be easily derived, which can be helpful for further control management. With the Floquet theory, one also requires the least effort to model seasonality over some climatic parameters that can naturally be shown from the behavior of the data. The results from this research may provide local authorities with small time to take necessary decisions and actions to safeguard the situation. However, since the results completely rely on the internal processes of the model, the stability results can be different from model to model. The authorities may blame the modeler directly once the prediction of outbreaks derived from a single model misinterpret the actual outbreaks. Caution should be taken when interpreting the analytical results of the model by its uniqueness, unless we can produce a more reasonable model. 

Further development of the model, which will most likely hinder from obtaining analytically sound results, could include the dynamics of host and vector in the six subregions. This will allow for a better measurement over the spatial distribution of populations. Using such a network model would allow the flow of people for work or regular visits to other zones to be modeled, as this has been considered an important factor in the spread of dengue \cite{NS2014,BDO2016,FSM2005}. Another interesting issue to address when investigating network-based models is modeling the human mobility between zones if mobility data are not known. Previous studies proposed the so-called radiation model \cite{SGM2012,SGM2016} and gravity model \cite{MSJ2013} to model human mobility on a network. However, by more zones the number of dynamic states can easily grow, which significantly increases the computational effort required for data assimilation. Notwithstanding the simplicity of the existence and stability theorems from Floquet theory, this model can less likely elicit explicit stability analysis results.

For more reasonable forecasts, the magnitude of $\beta$ can be extended by revealing its close relationship with the available data of extrinsic factors. This basically aims at characterizing the relationship between dengue incidence and the extrinsic factors. Further work might include collecting reasonable extrinsic factors that contribute to the emergence of dengue; for example, water precipitation, air humidity, wind speed, visibility, and coastal temperature (related to the occurrence of El Ni\~{n}o). Now the question remains how one can relate $\beta$ with those extrinsic factors, for which we proposed a simple idea in this paper. Let $E_k$, $k=1,\cdots,K$, denote the chosen extrinsic factors and let $\beta$ be expressed as $\beta=\alpha+\sum_{j=1}^m\delta_j\beta_j$ where $\beta_j=\cos(2\pi\omega_j t)$. The basic idea lies in the normalization of all the aforementioned variables such that they heavily correlate, because then a linear relationship can be inferred; see Fig.~\ref{fig:transform}. 

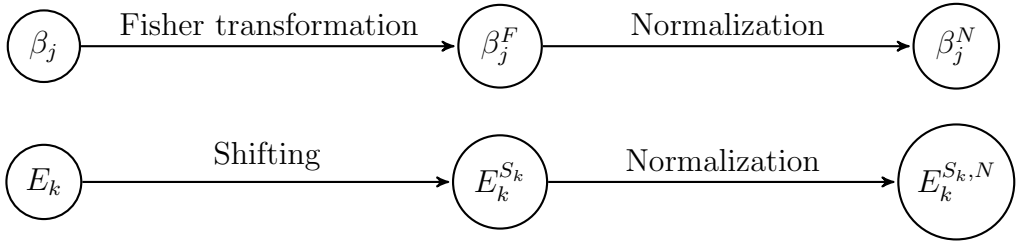
\begin{figure}[htbp!]
\centering
\begin{tikzpicture}[->,>=stealth',shorten >=1pt,auto,
                thick,main node/.style={circle,draw}]

  \node[main node] (1) {$\beta_j$};
  \node[main node] (2) [right of=1,xshift=5cm] {$\beta^F_j$};
  \node[main node] (3) [right of=2,xshift=5cm] {$\beta^N_j$};
  \node[main node] (4) [below of=1,yshift=-0.8cm] {$E_k$};
  \node[main node] (5) [right of=4,xshift=5cm] {$E^{S_k}_{k}$};
  \node[main node] (6) [right of=5,xshift=5cm] {$E^{S_k,N}_{k}$};

  \path
    (1) edge [] node [above] {Fisher transformation} (2)
    (2) edge node [above] {Normalization} (3)
    (4) edge node [above] {Shifting} (5)
    (5) edge node [above] {Normalization} (6);      
\end{tikzpicture}
\caption{Transformations of the optimal $\beta_j$ and all the collected extrinsic factors $E_k$.}
\label{fig:transform}
\end{figure}

Since $\beta_j\in [-1,1]$ for all time, the so-called Fisher transformation $\beta_j\mapsto\frac{1}{2}\log \left\lvert\frac{1+\beta_j}{1-\beta_j}\right\rvert=:\beta^F_j$ helps map $\beta_j$ into another variable $\beta_j^F$ that is normally distributed. According to \cite{BB1999}, $\beta^F_j\stackrel{\text{i.i.d.}}{\sim}N\left(\frac{1}{2}\log\left\lvert \frac{1+\rho_j}{1-\rho_j}\right\rvert,\frac{1}{n-3}\right)$, where $\rho_j$ denotes the Pearson product moment correlation coefficient and $n$ denotes the number of discrete time points used. Subtracting $\beta^F_j$ by its mean and dividing it by its standard deviation gives us the normalized variable $\beta^N_j\stackrel{\text{i.i.d.}}{\sim}N\left(0,1\right)$. As far as any extrinsic factor $E_k$ is concerned, our aim is to find a time shift such that the shifted normalized version of $E_k$, $E^{S_k,N}_{k}\stackrel{\text{i.i.d.}}{\sim}N\left(0,1\right)$, achieves its highest correlation with the normalized $\beta$. Previous studies suggest that there should be some time delay before any extrinsic factor leads to the most favorable condition for the mosquitoes to breed \cite{HRW2012,WJJ2013,WG2017}. With all the variables being normalized, we propose the following ansatz
\begin{equation*}
\beta^N_j \approx \sum_{k=1}^K\varepsilon_{jk}E^{S_k,N}_{k}\quad\text{where }\sum_{k=1}^K\varepsilon_{jk}^2=1\text{ for each }j,
\end{equation*}
according to the properties of the normal distribution. Now the problem remains to find which $\varepsilon_j=(\varepsilon_{jk})_{k=1}^K$ solves the following problem (for each $j$)
\begin{equation*}
\begin{aligned}
& \underset{\varepsilon_j\in\R^K}{\text{minimize}}
& &\frac{1}{2}\left\lVert \beta^N_j- \sum_{k=1}^K\varepsilon_{jk}E^{S_k,N}_{k}\right\rVert^2,\\
& \text{subject to}
& & \sum_{k=1}^K\varepsilon_{jk}^2=1.
\end{aligned}
\end{equation*}
Finally, restoring everything to the formulation of $\beta$, we obtain the following model
\begin{equation*}
\beta=\alpha+\sum_{j=1}^m\delta_j\frac{\exp(2\beta^F_j)-1}{\exp(2\beta^F_j)+1} \quad\text{where }\beta^F_j=\frac{1}{2}\log\left\lvert \frac{1+\rho_j}{1-\rho_j}\right\rvert + \sqrt{\frac{1}{n-3}}\sum_{k=1}^K\varepsilon_{jk}E^{S_k,N}_{k}.
\end{equation*}
We did not aim in this study to perform the last task, which is left for future work. As a warm-up to this, including the coastal temperature data might allow the underestimation of the outbreak produced by the last scheme as in Fig.~\ref{fig:assimilmultR} to be corrected.

\section*{Acknowledgments}
This research was financially supported by the Indonesia Ministry of Research and Higher Education (Kemenristek DIKTI), with PUPT research grant scheme 2017. The authors would like to show gratitude to Kartika Anggun Dimar Setio from the Faculty of Public Health, University of Indonesia for providing access to the dengue incidence data from the Jakarta Health Office, and to anonymous reviewers for their comments that greatly improved the manuscript. Part of the contents in this paper was presented in the project seminar at the University of Koblenz by the third author.


\section*{Appendix}
\appendix

\section{Proof of Theorem~\ref{thm:existperiodic}}
\label{app:exist}
Let $x_0\in U_1(x^{\ast})$ and $\delta>0$ be defined as to hold $|\delta|<\alpha_{1}$ for some $\alpha_1$ such that a unique solution $x=x(t,x_0,\delta)$ exists. 
Let $S( x_0,\delta):= x(\sigma,x_0,\delta)- x_0$. Since the whole vector field $f$ is continuous in time, state, and on $(-\alpha_{1},\alpha_{1})$, $S$ is continuously differentiable there. Observe that $S(x^{\ast},0)=0$ by the definition of the equilibrium. We have to calculate $\nabla_{ x_0}S(x^{\ast},0)$ for the interest of investigating the eigenvalues of the gradient of the autonomous function $g$, yet further explanation settles. To do so, we require investigating the function $\zeta(t):=\nabla_{x_0} x(t, x^{\ast} ,0)$. We note that $\nabla_{x_0} x(t, x_0,\delta)$ satisfies the following variational equation
\begin{equation*}
\frac{\dd}{\dd t}\nabla_{ x_0} x(t, x_0,\delta)=\nabla_{ x_0} f(t, x(t, x_0,\delta),\delta)=\nabla_{ x} f(t, x(t, x_0,\delta),\delta)\nabla_{ x_0} x(t, x_0,\delta),
\end{equation*}
with $\nabla_{ x_0} x(0, x_0,\delta)=\text{id}$ from which $\zeta$ satisfies
\begin{equation*}
\frac{\dd}{\dd t}\zeta=\nabla_{ x} g(x^{\ast})\zeta \quad \text{with }\zeta(0)=\text{id}.
\end{equation*}
The last result gives us the desired $\nabla_{x_0}S(x^{\ast},0)=\zeta(\sigma)-\text{id}=\exp(\nabla_{ x} g(x^{\ast})\sigma)-\text{id}$. Let $v$ be an eigenvector of $\nabla_{ x} g(x^{\ast})$ that associates with an eigenvalue $\lambda$. The Taylor expansion for matrix exponentials gives us $(\exp(\nabla_{ x}g(x^{\ast})\sigma)-\text{id}) v=(\exp(\sigma\lambda)-1) v$, which allows us to know that $\exp(\sigma\lambda)-1$ is an eigenvalue of $\nabla_{x_0}S(x^{\ast},0)$. Now, in case all eigenvalues $\lambda$ of $\nabla_{ x} g(x^{\ast})$ satisfy $\sigma\lambda \notin 2\pi i\mathbb{Z}$, we immediately assure that $\det(\nabla_{x_0}S(x^{\ast},0))=\prod_{\lambda}\exp(\sigma\lambda)-1\neq 0$, there $\nabla_{ x_0}S(x^{\ast},0)$ has some bounded inverse. By the implicit function theorem, there exist a domain $U_2(x^{\ast})\times(-\alpha_{2},\alpha_{2})$ and a continuously differentiable function $x_0(\delta)$ for $(\delta, x_0(\delta))$ defined on this domain such that $S(x_0(\delta),\delta)=0$ or eventually $ x(\sigma, x_0(\delta),\delta)= x_0(\delta)$. Since $f$ is $\sigma$-periodic over time, then $x(t+\sigma, x_0(\delta),\delta)= x(t, x_0(\delta),\delta)$ if and only if $x(\sigma, x_0(\delta),\delta)= x_0(\delta)$. The desired domain follows from letting $U(x^{\ast}):=U_1(x^{\ast})\cap U_2(x^{\ast})$ and $\alpha$ such that $(-\alpha,\alpha)\subset (-\alpha_{1},\alpha_{1})\cap(-\alpha_{2},\alpha_{2})$.

\section{Proof of Theorem~\ref{thm:stableperiodic}}
\label{app:stable}
By specifying $y:= x-\phi$ and $m(t, y) :=  f(t, y+\phi)- f(t,\phi)-A(t) y$ where $A(t):=\nabla_{ x} f(t,\phi)$, one assures that $y$ satisfies $\dot{y}=A(t) y+ m(t, y)$ 
where $ m(t,0)=0$ and $\nabla_{ y} m(t,0)=0$ by the chain rule. Observe that $m$ is $\sigma$-periodic over time. Let $Z(t)$ be the fundamental matrix that corresponds to the system matrix $A(t)$. Calculating the solution $y$ we obtain the following abstraction
\begin{align*}
y(t)&= Z(t) y_0+\int_0^tZ(t)Z(s)^{-1} m(s, y)\,\dd s\\
&=P(t)\exp(Qt)y_0 + \int_0^tP(t)\exp(Q(t-s))P(s)^{-1}m(s,y)\,\dd s.
\end{align*}
We know that $P(t)$ is continuously differentiable, which leads to its boundedness. We also know that $P(t)$ is uniformly nonsingular due to the nonsingularity of the fundamental matrix. The fact that multiplications and additions of some continuous periodic functions with uniperiod return a continuous periodic function guarantees that the determinant of $P(t)$ is uniformly continuous and $\sigma$-periodic. Together with continuous $\sigma$-periodic elements of the adjoint matrix of $P(t)$ gives the continuous $\sigma$-periodic inverse $P(t)^{-1}$. Let us denote $C_1$ and $C_2$ as upper bounds for $\lVert P(t)\rVert$ and $\max_{0\leq s\leq t}\lVert P(s)^{-1}\rVert$, respectively, for all $t\geq 0$, which make sense due to the continuity of matrix norms. Let us denote by $-\lambda$, for which $\lambda>0$, an upper bound for all the eigenvalues of $Q$, provided that \ref{item:floquet} is satisfied. This means that there exists a constant $C_3$ such that $\lVert\exp(Q t)\rVert\leq C_3\exp(-\lambda t)$. For the sake of making a contradiction, let us assume that $\lim_{t\rightarrow\infty}\lVert y(t)\rVert>0$, because then $\min_{0\leq t\leq \infty}\lVert y(t)\rVert:=\lVert y^{\ast}\rVert$ is positive for all $t\geq 0$ due to the fact that once $y$ touches zero at any time, it will remain zero for any future time based on the governing equation for $y$. The aforementioned estimates altogether give us
\begin{align*}
\lVert y(t)\rVert &\leq C_1C_3\exp(-\lambda t)\lVert y_0\rVert + C_1C_2C_3\int_0^t\exp(-\lambda (t-s))\lVert m(s,y)\rVert\,\dd s\\
&\stackrel{\ref{item:m}}{\leq} C_1C_3\exp(-\lambda t)\lVert y_0\rVert + \frac{C_1C_2C_3M\varepsilon}{\lVert y^{\ast}\rVert}\int_0^t\exp(-\lambda (t-s))\lVert y(s)\rVert\,\dd s.
\end{align*}
By moving the factor $\exp(-\lambda t)$ to the left-hand side of the very last inequality, Gronwall's inequality in the integral form for the function $\lVert y(t)\rVert\exp(\lambda t)$ for which $C_1C_3\lVert y_0\rVert $ is nondecreasing reveals $\exp(\lambda t)\lVert  y(t)\rVert\leq C_1C_3\lVert y_0\rVert \exp\left(\frac{C_1C_2C_3M\varepsilon}{\lVert y^{\ast}\rVert} t\right)$ or
\begin{align*}
\lVert y(t)\rVert &\leq C_1C_3\lVert y_0\rVert\exp\left(\left(\frac{C_1C_2C_3M\varepsilon}{\lVert y^{\ast}\rVert}-\lambda\right)t\right).
\end{align*}
The key idea here was to include $\varepsilon$ as an upper bound for $\lVert y(s)\rVert$ on $[0,t]$ from which $\varepsilon$ and $y_0$ can now become control parameters. By selecting $\varepsilon$ and $\lVert y_0\rVert$ small enough such that $C_1C_3\lVert y_0\rVert\leq \varepsilon$ and $C_1C_2C_3M\varepsilon\slash \lVert y^{\ast}\rVert-\lambda<0$, we yield the fact that
\begin{equation*}
\lVert y(t)\rVert\leq \varepsilon\exp\left(\left(\frac{C_1C_2C_3M\varepsilon}{\lVert y^{\ast}\rVert}-\lambda\right)t\right)
\end{equation*}
holds for all $t$ provided that $\lVert y(t)\rVert\leq \varepsilon$. Observe that as long as the boundedness assumption $\lVert y(t)\rVert\leq \varepsilon$ is kept from the starting time $0$ up to any positive future time, any choice of the initial condition $\lVert y_0\rVert\leq \min\{\varepsilon\slash(C_1C_2),\varepsilon\}$ rediscovers $\lVert y(t)\rVert\leq\varepsilon\exp\left(\left(\frac{C_1C_2C_3M\varepsilon}{\lVert y^{\ast}\rVert}-\lambda\right)t\right)\leq \varepsilon$. Therefore, the boundedness by an exponentially decreasing function must hold on $[0,\infty)$, which is a contradiction to the given assumption. Therefore, $\lim_{t\rightarrow\infty}\lVert y(t)\rVert = 0$. 

\section{Proof of Theorem~\ref{thm:existIR}}
\label{app:existIR}
The Jacobian $\nabla_x g(\DFE)$ clearly gives us the eigenvalues $\alpha\slash\nu-\gamma-\mu$ and $-\mu-\kappa$ that can never be pure imaginary. The Jacobian $\nabla_x g(\EE)$ has the characteristic polynomial $\lambda^2+a\lambda+b$ where
\begin{align*}
a &= \frac{(\mu+\kappa)((\mu+\gamma)(\alpha-\gamma\nu) + \mu(\alpha-\gamma\nu-\mu\nu) + \alpha^2+\alpha\kappa+\alpha\kappa\nu)}{\beta(\gamma\nu + (\kappa+\mu)(1+\nu))}\\
b&=\frac{(\mu+\kappa)((\alpha+\mu)(\gamma+\kappa+\mu) + \gamma\kappa)(\alpha-\mu\nu-\gamma\nu)}{\beta(\gamma\nu + (\kappa+\mu)(1+\nu))}.
\end{align*}
If $\RM_0>1$, then we can assure that $b>0$ and
\begin{equation*}
a\geq \frac{(\mu+\kappa)((\mu+\gamma)(\alpha-\gamma\nu-\mu\nu) + \mu(\alpha-\gamma\nu-\mu\nu) + \alpha^2+\alpha\kappa+\alpha\kappa\nu)}{\beta(\gamma\nu + (\kappa+\mu)(1+\nu))}>0.
\end{equation*}
Since the addition of the two roots of the polynomial gives a real negative value and the multiplication gives a real positive value, the roots cannot be pure imaginary, but can be complex conjugate with negative real part.

\section{Proof of Theorem~\ref{thm:stableIR}}
\label{app:stableIR}
The proof reads as follows.
\begin{enumerate}[label=\normalfont ({B}\arabic*)]
\item According to Theorem~\ref{thm:stableperiodic}, we must first reveal the explicit formulation of $A(t)$, which is, by $\phi\equiv 0$,
\begin{equation*}
A(t) = \left(\begin{array}{cc}
\frac{\alpha}{\nu}+\frac{\delta p(t)}{\nu}-\gamma-\mu & 0\\
\gamma & -\mu-\kappa
\end{array}
\right).
\end{equation*}
We can simply decompose the following antiderivative matrix as the sum of two matrices 
\begin{equation*}
\int_0^t A(s)\,\text{d}s= \underbrace{\left(\begin{array}{cc}
\frac{\delta}{\nu}\int_0^tp(s)\,\text{d}s & 0\\
0 & 0
\end{array}
\right)}_{:=B(t)} + \underbrace{\left(\begin{array}{cc}
\frac{\alpha }{\nu}-\gamma -\mu  & 0\\
\gamma  & -\mu  -\kappa 
\end{array}
\right)}_{:=Q}t
\end{equation*}
because then the corresponding fundamental matrix reads as
\begin{equation*}
Z(t) = \exp(B(t))\exp(Qt) = \underbrace{\left(\begin{array}{cc}
\exp\left(\frac{\delta}{\nu}\int_0^tp(s)\,\text{d}s\right) & 0\\
0 & 1
\end{array}
\right)}_{:=P(t)}\exp(Qt).
\end{equation*}
Observe that $P(0)=P(\sigma)=\text{id}$. If $\RM_0<1$, then the Floquet exponents (i.e.,\ the eigenvalues of $Q$) are negative and real, which fulfills the condition~\ref{item:floquet}. Finally, a glimpse over the following formulation
\begin{equation*}
m(t,y)=\left(
\frac{\beta(t)(N-y_1-y_2)y_1}{y_1+\nu N} - (\gamma+\mu)y_1 - \left(\frac{\beta(t)}{\nu}-\gamma-\mu\right)y_1,0
\right)
\end{equation*}
together with the fact that $y_1,y_2\geq 0$ while $y_1+y_2\leq N$ give us
\begin{align*}
\lVert m(t,y)\rVert&\leq \underbrace{\left(\left\lVert\frac{\beta(t) N}{y_1+\nu N}\right\rVert+ \left\lVert\frac{y_1}{y_1+\nu N}\right\rVert+\left\lVert\frac{\beta(t)}{\nu}\right\rVert\right)}_{\leq M_1} y_1+\underbrace{\left\lVert\frac{y_1}{y_1+\nu N}\right\rVert}_{\leq M_2} y_2\\
&\leq \max\{M_1,M_2\}\sqrt{y_1^2+2y_1y_2+y_2^2}\leq \underbrace{\sqrt{2}\max\{M_1,M_2\}}_{:=M}\sqrt{y_1^2+y_2^2},
\end{align*}
which confirms the condition~\ref{item:m}.
%

\item As far as the nontrivial periodic solution $\phi$ is concerned, we obtain 
\begin{align*}
A(t) &= \left( \begin {array}{cc} 
-{\frac {\beta\,\phi_{{1}}}{\phi_{{1
}}+ \nu N}}+{\frac { \beta\left( N-\phi_{{1}}-\phi_{{2}} \right) }{
\phi_{{1}} + \nu N}} \left(1- {\frac {\phi_{{1}}}{\phi_{{1
}}+ \nu N}}\right)-\gamma-\mu & -{
\frac {\beta\,\phi_{{1}}}{\phi_{{1}} + \nu N}}\\ \noalign{\medskip}
\gamma&-\mu-\kappa\end {array} \right)\\
m(t,y) &= \left( \begin {array}{c} -{\frac {\beta\,N\nu\,y_{{1}} \left( N\nu\,y
_{{1}}+N\nu\,y_{{2}}+Ny_{{1}}+\phi_{{1}}y_{{2}}-\phi_{{2}}y_{{1}}
 \right) }{ \left( \phi_{{1}}+y_{{1}}+\nu N \right)  \left(
\phi_{{1}}+\nu N \right) ^{2}}}\\ \noalign{\medskip}0\end {array} \right)
\end{align*}
from which the condition~\ref{item:m} applies to $m(t,y)$. We may decompose $\beta(t)$ to decompose $\int_0^{t}A(s)\,\dd s$ and 
choose the following specification
\begin{align*}
\tilde{Q}&=\left(\begin{array}{cc}
\frac{\alpha}{t}\int_0^t -\frac{\phi_1(s)}{\phi_1(s)+\nu N} + \frac{N-\phi_1(s)-\phi_2(s)}{\phi_1(s)+\nu N}\left(1-\frac{\phi_1(s)}{\phi_1(s)+\nu N}\right) \,\dd s- \gamma  - \mu  & -\frac{\alpha}{t}\int_0^{t}\frac{\phi_1(s)}{\phi_1(s)+vN}\,\dd s\\
\gamma & -\mu -\kappa 
\end{array}
\right),
\end{align*}
which is dependent on time. With this specification, we obtain
\begin{align*}
\tilde{B}(t)&=\left(\begin{array}{cc}
\frac{1}{t}\int_0^t p(s)\left(-\frac{\phi_1(s)}{\phi_1(s)+\nu N} + \frac{N-\phi_1(s)-\phi_2(s)}{\phi_1(s)+\nu N}\left(1-\frac{\phi_1(s)}{\phi_1(s)+\nu N}\right)\right) \,\dd s  & -\frac{1}{t}\int_0^{t}\frac{p(s)\phi_1(s)}{\phi_1(s)+vN}\,\dd s\\
0 & 0 
\end{array}
\right)
\end{align*}
and $\tilde{P}(t) :=\exp(\tilde{B}(t))$ where it holds that $Z(t)=\tilde{P}(t)\exp(\tilde{Q}t)$. Observe that all the integrands in every element of $\tilde{B}(t)$ are continuous and $\sigma$-periodic. 
This shows that $\tilde{P}(t)$ is bounded. A similar argument as in the proof of Theorem~\ref{thm:stableperiodic} gives us an idea to see that the stability of the nontrivial periodic solution is determined by the eigenvalues of $\tilde{Q}$. It turns out that if the $(1,1)$ element of $\tilde{Q}$ is negative, then the resulting characteristic polynomial of $\tilde{Q}$ exhibits a positive coefficient of the first-order term and a positive constant. This simply gives us two eigenvalues that lie in the open left-half plane in $\mathbb{C}$. It now remains to show under what condition the $(1,1)$ element of $\tilde{Q}$ remains negative for all time $t\geq 0$. The mean value theorem for integrals discovers that for every $t>0$, there exists $\tau\in [0,t]$ such that the $(1,1)$ element of $\tilde{Q}$ is equal to
\begin{equation*}
\underbrace{\left(-\frac{\nu\phi_1(\tau)}{\phi_1(\tau)+\nu N} + \frac{\nu(N-\phi_1(\tau)-\phi_2(\tau))}{\phi_1(\tau)+\nu N} \left(1-\frac{\phi_1(\tau)}{\phi_1(\tau)+\nu N} \right)\right)}_{:=\Phi(\tau)}\frac{\alpha}{\nu t}\int_0^t1\,\dd s-\gamma-\mu.
\end{equation*}
Since $\Phi$ is continuous and $\sigma$-periodic, it suffices to analyze the negativity on $[0,\sigma]$. In particular, if
\begin{equation*}
\RM_{\max}:=\max_{0\leq t\leq \sigma}\frac{\alpha}{(\gamma+\mu)\nu}\Phi(t)<1,
\end{equation*}
then, certainly, the $(1,1)$ element of $\tilde{Q}$ is always negative. 
\end{enumerate}
\end{document}